\documentclass[USenglish]{article}

\usepackage[utf8]{inputenc}
\usepackage{microtype}

\usepackage{amsmath,amssymb,amsthm}
\usepackage{graphicx}
\usepackage{color}
\usepackage{algorithm}
\usepackage{algpseudocode}
\usepackage{textcomp}
\usepackage{float}
\usepackage{hyperref}
\usepackage[caption = false]{subfig}
\usepackage{cleveref}
\usepackage{tikz}
\usepackage{appendix}
\usepackage[left=3cm,right=3cm,top=1.5cm,bottom=2cm,includeheadfoot]{geometry}

\newtheorem{theorem}{Theorem}[section]

\newtheorem{lemma}{Lemma}[section]
\newtheorem{proposition}{Proposition}[section]

\theoremstyle{definition}
\newtheorem{definition}{Definition}
\newtheorem{remark}{Remark}
\newtheorem{example}{Example}

\newcommand{\R}{\mathbf{R}}

\newcommand{\weakstarto}{\stackrel*\rightharpoonup}
\newcommand{\hd}{\mathcal{H}}

\newcommand{\leb}{\mathcal{L}}
\newcommand{\restr}{{\mbox{\LARGE$\llcorner$}}}

\newcommand{\pushforward}[2]{{{#1}_{\#}#2}}

\DeclareMathOperator{\dive}{div}
\DeclareMathOperator*{\argmin}{argmin}
\DeclareMathOperator*{\argmax}{argmax}
\DeclareMathOperator{\arsinh}{arsinh}
\DeclareMathOperator*{\bigtimes}{\raisebox{-.6ex}{\scalebox{2}{$\times$}}}

\newcommand{\E}{\mathcal{E}}
\renewcommand{\d}{{\mathrm d}}
\newcommand{\dist}{{\mathrm{dist}}}
\newcommand{\flux}{{\mathcal{F}}}
\newcommand{\dom}{\overline{\Omega}\times[0,M]}
\newcommand{\G}{\mathcal{G}}
\newcommand{\D}{\mathcal{D}}
\newcommand{\K}{\mathcal{K}}
\newcommand{\A}{\mathcal{A}}
\newcommand{\C}{\mathcal{C}}
\newcommand{\XX}{\mathcal{X}}
\newcommand{\Y}{\mathcal{Y}}
\newcommand{\manifold}{\mathcal{N}}
\newcommand{\M}{\mathcal{M}}
\newcommand{\BV}{\mathrm{BV}}
\newcommand{\N}{\mathcal{N}}
\newcommand{\T}{\mathcal{T}}
\newcommand{\dofnodes}{\N'}

\newcommand{\notinclude}[1]{}

\newlength{\dhatheight}

\begin{document}



\title{An adaptive finite element approach for lifted branched transport problems}
\author{Carolin Dirks, Benedikt Wirth}

\maketitle
\abstract{We consider so-called branched transport and variants thereof in two space dimensions.
In these models one seeks an optimal transportation network for a given mass transportation task.
In two space dimensions, they are closely connected to Mumford--Shah-type image processing problems,
which in turn can be related to certain higher-dimensional convex optimization problems via so-called functional lifting.
We examine the relation between these different models and exploit it to solve the branched transport model numerically via convex optimization.
To this end we develop an efficient numerical treatment based on a specifically designed class of adaptive finite elements.
This method allows the computation of finely resolved optimal transportation networks
despite the high dimensionality of the convex optimization problem and its complicated set of nonlocal constraints.
In particular, by design of the discretization the infinite set of constraints reduces to a finite number of inequalities.
}




\section{Introduction}

During the past two decades a class of models has been developed that can be interpreted as variants of classical optimal transport (more specifically Wasserstein-1 transport).
Given two nonnegative (probability) measures, a material source $\mu_+$ and a material sink $\mu_-$,
one needs to transport the material from $\mu_+$ to $\mu_-$ at minimal cost.
The underlying cost functionals have the feature that the cost per transport distance is not proportional to the amount of transported mass.
Instead, a subadditive cost function penalizes transport of small masses disproportionately stronger
and thus promotes mass aggregation and transport of the accumulated material along an emerging common transport network.
The resulting networks exhibit a complicated branching structure, where the grade of ramification and the network geometry are controlled by the precise form of the cost functional.
Particular instances of this model class include the so-called branched transport \cite{Xia2003,MaSoMo2003}, urban planning \cite{BrBu2005}, and the Steiner tree problem \cite{GiPo1968}
(note that there is a large variety of possible model formulations which in the end turn out to be equivalent, see \cite{BrWi2018} and the references therein).
There exist a variety of interesting applications such as the optimization of communication or public transportation networks \cite{Gi67,BuPrSoSt09}
or the understanding of vascular structures in plants and animals \cite{Xi07,BrSu18,XiCrFa16}, to name just a few.

Typically, the corresponding energy landscape is highly non-convex.
Consequently, the identification and construction of a globally optimal transportation network is a challenging task.
In this work we exploit a connection of the two-dimensional transportation network problem to convex image processing methods
in order to compute globally optimal network geometries numerically.
We already made use of this connection in previous work \cite{BrRoWi2018} to prove lower bounds on the transportation cost and to perform preliminary numerical simulations,
however, since our sole interest were lower bounds, we had neither fully understood the underlying connection nor come up with an efficient, tailored numerical scheme.
In contrast, in the present work our focus is on numerically solving two-dimensional branched transport problems.
For this we will prove the equivalence of the original transportation network problem to a sequence of models leading to a convex image inpainting problem
(the only gap in this sequence of equivalences will be a classical relaxation step for the Mumford--Shah functional,
whose tightness is unto this date not known to the best of our knowledge).
Even though the final problem is convex, it features a high dimensionality and a huge number of constraints which render its solution with standard methods infeasible.
We thus proceed to design a particular adaptive discretization which tremendously decreases the computational effort
and thereby allows computation of highly resolved optimal transportation schemes.

\subsection{Existing numerical methods for branched transport-type problems}

In order to simulate optimal transportation networks, several approaches have been investigated in the literature.
Based on a Eulerian formulation via mass fluxes,
Xia introduced an initial approach for numerically finding an optimal graph between two measures \cite{Xia2003,Xia2015}.
This local optimization technique was extended to a minimization algorithm in \cite{Xia2008},
which in several numerical examples with a single source point and a fixed number of $N$ sinks seems to yield almost optimal networks.
It was shown in \cite{Xia2008} and \cite{Xia2015} that, although not necessarily leading to a global minimizer,
this optimization algorithm provides an approximately optimal transport network and is applicable even in case of a large number of sinks ($N\approx400$).
Two heuristic approaches based on stochastic optimization techniques on graphs were presented in \cite{MaMiSc2012} and \cite{Pi2014}.
As before, these method are capable of providing almost optimal network structures, but cannot guarantee global optimality either.
The limit case of the Steiner tree problem, where the transport cost is independent of the amount of transported mass, was treated more extensively in the literature.
Due to the independence of the transported mass, there exist very efficient algorithms in a planar geometry
providing a globally optimal Steiner tree (see for instance the GeoSteiner method \cite{JuWaWiZa2018} or Melzak's full Steiner tree algorithm \cite{Me1961}).
For more than two space dimensions there exist fewer approaches which are less efficient;
an overview of some methods for the Steiner tree problem in $n$ dimensions is provided in \cite{FaLeMa16},
where the main ideas trace back to \cite{GiPo1968,Sm1992,MaMiXa2000,FoBrWiZa2014}.
For the general transportation network problem, a widely used approach was inspired by elliptic approximations of free-discontinuity problems
in the sense of Modica--Mortola and Ambrosio--Tortorelli via phase fields.
In \cite{OuSa2011,ChMeFe2016,Mo2017,Fe18,FeRoWi2019,Wi19},
corresponding phase field approximations have been presented for the classical branched transport problem, the Steiner tree problem, a variant of the urban planning problem (which is piecewise linear in the amount of transported mass) or more general cost functions,
however, all restricted to two space dimensions.

\subsection{Contributions of our work}

In this work, we build on the approach introduced by \cite{BrRoWi2018}, which consists in a novel reformulation of the optimal transportation network problem as a Mumford--Shah-type image inpainting problem in two dimensions. Roughly speaking, the optimal network is represented by the rotated gradient of a grey-value image of bounded variation. The resulting equivalent energy functional resembles the structure of the well-known Mumford--Shah functional \cite{MuSh1989}, which in turn admits a convex higher-dimensional relaxation by a so-called functional lifting approach \cite{AlBoDa03,PoCrBiCh2009}.

In a little more detail, fix some domain $\Omega\subset\R^2$ and consider a source and sink $\mu_+$ and $\mu_-$ supported on the boundary $\partial\Omega$.
We denote by $\tau(m)$ the cost for transporting mass $m$ along one unit distance.
Describing the transportation network as a vector measure $\flux\in\M(\overline\Omega;\R^2)$,
the associated generalized branched transport cost can be defined as some functional $\E(\flux)$ depending on the choice of $\tau$.
Now any vector measure $\flux\in\A_\flux=\{\flux\in\M(\overline\Omega;\R^2)\,|\,\dive\flux=\mu_+-\mu_-\}$
can be interpreted as the rotated gradient of an image $u\in\A_u$ for some set of admissible images $\A_u\subset\BV(\Omega)$
so that the generalized branched transport cost can be reformulated as the cost
\begin{equation*}
\tilde\E(u)=\int_{S_u\cap\overline\Omega}\tau([u])\,\d\hd^1+\tau'(0)|Du|(\Omega\setminus S_u)
\end{equation*}
of the associated image $u$ (where $S_u$ denotes the discontinuity set and $[u]$ the jump of $u$).
Writing $1_u$ for the characteristic function of the subgraph of $u$, Alberti et al.\ suggested in \cite{AlBoDa03} to rewrite $\tilde\E(u)$ as
\begin{equation*}
\G(1_u)=\sup_{\phi\in\K}\int_{\overline\Omega\times\R}\phi\cdot\,\d D1_u
\end{equation*}
for some particular set $\K$ of three-dimensional vector fields depending on $\tau$.
By convexifying the set of characteristic functions $1_u$ to a set $\C$ of more general functions $v:\Omega\times\R\to[0,1]$
one finally arrives at a convex optimization problem, whose dual can be used to provide a lower bound.
In summary, as proved rigorously in \cite{BrRoWi2018} we have
\begin{multline}\label{eqn:mainInequalities}
\inf_{\flux\in\A_\flux}\E(\flux)
\geq\inf_{u\in\A_u}\tilde\E(u)
\geq\inf_{u\in\A_u}\G(1_u)
\geq\inf_{v\in\C}\G(v)\\
\geq\sup_{\phi\in\K}\int_{\partial\Omega\times\R}1_{u(\mu_+,\mu_-)}\phi\cdot n \ \d\hd^2 - \int_{\Omega\times\R}\max\{0,\dive\phi\}\,\d x\,\d s,
\end{multline}
where $1_{u(\mu_+,\mu_-)}$ denotes a particular binary function defined on $\partial\Omega\times\R$.
The left-hand side of the above is the original generalized branched transport problem.
In \cite{BrRoWi2018} we used the right-hand side to prove lower bounds for $\E(\flux)$,
and we furthermore discretized this three-dimensional convex optimization problem via a simple finite difference scheme
and presented several simulation results for different scenarios.

From the viewpoint of numerics for branched transport problems, the results of \cite{BrRoWi2018} are unsatisfactory for two reasons:
(i) The final convex optimization problem was only shown to be a lower bound, whose solutions might actually differ from the minima of the original problem.
(ii) The employed numerical methods suffered from excessive memory and computation time requirements, rendering complex network optimizations infeasible.
The contribution of the present work is to remedy these shortcomings:
\begin{itemize}
\item
We prove equality for the whole above sequence of inequalities except for the third, $\inf_{u\in\A_u}\G(1_u)\geq\inf_{v\in\C}\G(v)$,
which we only conjecture to be an equality
(this is a particular instance of a long-standing, yet unsolved problem for which we can only provide some discussion and numerical evidence).
Note that while the equality $\inf_{u\in\A_u}\tilde\E(u)=\inf_{u\in\A_u}\G(1_u)$ might be considered known in the calculus of variations community
(it is for instance stated as Remark\,3.3 in the arXiv version of \cite{AlBoDa03}),
a rigorous proof was not available in the literature.
\item
For a set of simple example cases we provide fluxes $\flux\in\A_\flux$ and vector fields $\phi\in\K$
for which left- and right-hand side in the above inequality coincide (such vector fields are known as calibrations).
This serves the same two purposes as the original use of calibrations in \cite{AlBoDa03} for the classical Mumford--Shah functional:
It shows $\inf_{u\in\A_u}\G(1_u)=\inf_{v\in\C}\G(v)$ in various relevant cases,
and it provides explicit optimality results for particular settings of interest.
\item
We develop a non-standard finite element scheme,
allowing an efficient treatment of the lifted branched transportation network problem and providing locally high resolution results.
The main difficulties here lie in the high dimensionality due to the lifted dimension
as well as in the suitable handling of the infinite number of nonlocal inequality constraints defining the set $\K$.
The former difficulty is approached by the use of grid adaptivity, the latter by a particular design of the finite element discretization.
\end{itemize}

The above-mentioned equalities are presented and proved in \crefrange{thm:fluxImage}{thm:strongDuality} within \cref{sec:liftingTheory},
which also contains the calibration examples.
The tailored discretization and corresponding numerical algorithm are presented in \cref{sec:numerics} together with numerical results.


\subsection{Preliminaries}

Let us briefly fix some notation.
We denote by $\leb^n$ the $n$-dimensional Lebesgue measure, by $\hd^k$ the $k$-dimensional Hausdorff measure, and by $\delta_x$ the Dirac measure in a point $x\in\R^n$.
The space of $\R^N$-valued Radon measures on $\overline\Omega$ for $\Omega\subset\R^n$ some open bounded domain is denoted by $\M(\overline\Omega;\R^N)$.
For $N=1$, we write $\M(\overline\Omega)$ and define $\M_+(\overline\Omega)$ as the set of non-negative finite Radon measures on $\overline\Omega$.
For a measure $\flux\in\M(\overline\Omega;\R^N)$, the corresponding total variation measure and the total variation norm are denoted by $|\flux|$ and $\|\flux\|_\M=|\flux|(\overline\Omega)$, respectively.
The Radon measures can be viewed as the dual to the space of continuous functions, thus there is a corresponding notion of weak-* convergence, indicated by $\weakstarto$.
For a measure space $(X,\A,\mu)$ and some $Y\subset X$ with $Y\in\A$, the restriction of the measure $\mu$ onto $Y$ is written as $\mu\restr Y(A)=\mu(A\cap Y)$ for all $A\in\A$.
The Banach space of functions of bounded variation on $\Omega$, that is,
functions $u$ in the Lebesgue space $L^1(\Omega)$ whose distributional derivative is a vector-valued Radon measure,
is denoted $\BV(\Omega)$ with norm $\|u\|_{\BV} = \|u\|_{L^1} + \|Du\|_\M$.
The Banach space of continuous $\R^N$-valued functions on $\overline\Omega$ is denoted by $C^0(\overline\Omega;\R^N)$,
the space of compactly supported smooth $\R^N$-valued functions on $\Omega$ by $C_0^\infty(\Omega;\R^N)$.

For a convex subset $C$ of a vector space $X$ we write the orthogonal projection of $x\in X$ onto $C$ as $\pi_C(x) = \argmin_{y\in C}|x-y|$. The convex analysis indicator function of $C$ is denoted by $\iota_C(x)=0$ if $x\in C$ and $\iota_C(x)=\infty$ else.


\section{Functional lifting of the generalized branched transport cost}\label{sec:liftingTheory}

Below we briefly recapitulate the Eulerian formulation of the generalized branched transport problem in \cref{sec:branchedTransport},
after which we introduce the reformulation as Mumford--Shah image inpainting problem and its convexification via functional lifting in \cref{sec:FuncLifting}.
We will prove equivalence of the different resulting formulations except for one relaxation step, whose implications can only be discussed.
We then use the convex optimization problem to show optimality of a few particular network configurations in \cref{sec:calibrations}.

\subsection{Generalized branched transport}\label{sec:branchedTransport}

In generalized branched transport models, the cost for transporting a lump of mass $m$ along one unit distance is described by a transportation cost $\tau(m)$.
This transportation cost is taken to be subadditive, which encodes that transporting several lumps of mass together is cheaper than transporting each separately.
(Two further natural requirements from an application viewpoint are monotonicity and lower semi-continuity.)
For the purpose of this article we will restrict ourselves to the class of concave transportation costs (note that any concave function $\tau$ with $\tau(0)=0$ is subadditive),
which encompasses all particular models studied in the literature so far.

\begin{definition}[Transportation cost]\label{def:CostFunction}
A \emph{transportation cost} is a non-decreasing, concave, lower semi-continuous function $\tau:[0,\infty)\rightarrow[0,\infty)$ with $\tau(0)=0$. 
\end{definition}

\begin{example}[Branched transport, urban planning, and Steiner tree]\label{ex:BTandUP}
Three particular examples of transportation costs are given by
\begin{equation*}
\tau^{\mathrm{bt}}(m)=m^\alpha,\quad
\tau^{\mathrm{up}}(m)=\min\{am,m+b\},\quad
\tau^{\mathrm{st}}(m)=1\text{ if }m>0,\tau^{\mathrm{st}}(0)=0
\end{equation*}
for parameters $\alpha\in(0,1)$, $a>1$, $b>0$.
The original \emph{branched transport} model in \cite{Xia2003} and \cite{MaSoMo2003} uses $\tau^{\mathrm{bt}}$, and most analysis of transportation networks has been done for this particular case.
The \emph{urban planning} model, introduced in \cite{BrBu2005} and recast into the current framework in \cite{BrWi16}, is obtained for $\tau^{\mathrm{up}}$.
Here the material sources and sinks represent the homes and workplaces of commuters, and one optimizes the public transport network
($a$ has the interpretation of travel costs by other means than public transport, while $b$ represents network maintenance costs).
Finally, the \emph{Steiner tree} problem of connecting $N$ points by a graph of minimal length can be reformulated as generalized branched transport
by taking a single point as source of mass $N-1$ and the remaining $N-1$ points as sinks of mass $1$, using the transportation cost $\tau^{\mathrm{st}}$.
\end{example}

In the simplest formulation, the generalized branched transport problem is first introduced for simple transportation networks,
so-called discrete transport paths or discrete mass fluxes, which can be identified with graphs (see \cite{Xia2003,BrWi2018}).

\begin{definition}[Discrete mass flux]
Let $\mu_+=\sum_{i=1}^ka_i\delta_{x_i}$, $\mu_-=\sum_{j=1}^lb_i\delta_{y_j}$ be two measures with 
$x_i,y_j\in\R^n$, $a_i,b_j>0$. Let $G$ be a weighted directed graph in $\R^n$ with vertices $V(G)$, edges $E(G)$, and weight function $w:E(G)\rightarrow[0,\infty)$. For an edge $e\in E(G)$, we denote by $e^+,e^-$ its initial and final vertex and by $\vec{e}=\frac{e^--e^+}{|e^--e^+|}\in S^{n-1}$ its direction. Then the vector measure 
\begin{equation*}
\flux_G = \sum_{e\in E(G)} w(e)(\hd^1\restr e)\vec{e}
\end{equation*} 
is called a \emph{discrete mass flux}. It is a discrete mass flux \emph{between $\mu_+$ and $\mu_-$}, if $\dive\flux_G=\mu_+-\mu_-$ in the distributional sense.  
\end{definition} 

\begin{definition}[Discrete cost functional]
Let $\flux_G$ be a discrete mass flux \notinclude{between $\mu_+$ and $\mu_-$ }corresponding to a graph $G$. The \emph{discrete generalized branched transport cost functional} is given by 
\begin{align*}
\E(\flux_G) = \sum_{e\in E(G)}\tau(w(e)) \hd^1(e).
\end{align*}
\end{definition} 

In the above discrete setting, the weight function $w$ encodes the amount of mass flowing through an edge, while the distributional divergence constraint ensures that no mass is created or lost outside the source $\mu_+$ and sink $\mu_-$ of the mass flux.
Obviously, there can only be discrete mass fluxes between sources and sinks of equal mass.
For general mass fluxes, described as vector-valued measures, the cost is defined via weak-* relaxation. 

\begin{definition}[Continuous mass flux]
Let $\mu_+,\mu_-\in\M_+(\R^n)$. 
A vector measure $\flux\in\M(\R^n;\R^n)$ is a \emph{(continuous) mass flux between $\mu_+$ and $\mu_-$}, if $\dive\flux=\mu_+-\mu_-$ in the distributional sense. 
\end{definition} 

\begin{definition}[Continuous cost functional]\label{def:ContinuousCostFunctionals}
Let $\flux$ be a continuous mass flux\notinclude{ between $\mu_+$ and $\mu_-$}. The \emph{continuous generalized branched transport cost functional} is given by 
\begin{align*}
\E(\flux) = \inf\Big\{ \liminf_{k\rightarrow\infty} \ \E(\flux_{G_k}) \,\Big|\, (\flux_{G_k},\dive\flux_{G_k})\weakstarto(\flux,\dive\flux)\Big\}.
\end{align*}
\end{definition} 

Existence of minimizing mass fluxes between arbitrary prescribed sources $\mu_+$ and sinks $\mu_-$  has been shown in \cite{Xia2003,BrWi2018}
under growth conditions on the transportation cost $\tau$ near zero.

\subsection{Reformulation as an image inpainting problem in 2D and convexification}\label{sec:FuncLifting}

In \cite{BrRoWi2018} we introduced a reformulation of the branched transportation energy as an image inpainting problem in two space dimensions, leading to a convexification via a functional lifting approach and to the sequence \eqref{eqn:mainInequalities} of inequalities. Here we recall the key steps of this analysis, complement it with the derivation of the opposite inequalities, and finally derive the lifted convex optimization problem which will later form the basis of our numerical simulations.

From now on, let $\Omega\subset\R^2$ be open and convex (the following could easily be generalized to Lipschitz domains which would just lead to a more technical exposition),
and let $\mu_+,\mu_-\in\M_+(\partial\Omega)$ with equal mass $\|\mu_+\|_\M=\|\mu_-\|_\M$ denote a material source and sink supported on the boundary $\partial\Omega$.
We furthermore abbreviate $V=B_1(\Omega)\subset\R^2$ to be the open $1$-neighbourhood of $\Omega$,
whose sole purpose is to allow defining boundary values for images $u$ on $\Omega$ by fixing $u$ on $V\setminus\Omega$
(which is notationally easier than working with traces of $\BV$ functions).

\begin{remark}[Existence of optimal mass fluxes]
In the two-dimensional setting with $\mu_+$ and $\mu_-$ concentrated on the boundary $\partial\Omega$
one always has existence of optimal (that is, $\E$-minimizing) mass fluxes between $\mu_+$ and $\mu_-$, independent of the choice of $\tau$.
Indeed, there exists a mass flux of finite cost
(for instance a mass flux concentrated on $\partial\Omega$ which moves the mass round counterclockwise
and whose cost can be bounded from above by $\tau(\|\mu_+\|_\M)\hd^1(\partial\Omega)$)
so that existence of minimizers follows from \cite[Thm.\,2.10]{BrWi2018}.
\end{remark}

For an image $u\in \BV(V)$, one can define a mass flux $\flux_u\in\M(\overline\Omega;\R^2)$ as the rotated gradient of $u$,
\begin{equation*}
\flux_u = Du^\perp\restr\overline\Omega = (\nabla u^\perp \leb^2\restr V + [u]\nu_u^\perp\hd^1\restr S_u + D_cu^\perp)\restr\overline\Omega.
\end{equation*}
Here, $\nabla u$ denotes the approximate gradient of the image $u$, $S_u$ the approximate discontinuity set, $\nu_u$ the unit normal on $S_u$, $[u]=u^+-u^-$ the jump in function value across $S_u$ in direction $\nu_u$, $D_cu$ a Cantor part (see for instance \cite[\S\,3.9]{AmFuPa2000}), and $\perp$ counterclockwise rotation by $\frac\pi2$. Since $Du$ as a gradient is curl-free, $\flux_u$ is divergence-free (in the distributional sense) in $\Omega$.
It is now no surprise that fluxes between $\mu_+$ and $\mu_-$ correspond to images with particular boundary conditions.
To make this correspondence explicit, let $\gamma:[0,\hd^1(\partial\Omega))\rightarrow \partial\Omega$ be a counterclockwise parameterization of $\partial\Omega$ by arclength,
where without loss of generality we may assume $\gamma(0)=0\in\partial\Omega$, and abbreviate $\partial\Omega_t=\gamma([0,t))$.

\begin{definition}[Admissible fluxes and images]
Given $\mu_+,\mu_-\in\M_+(\partial\Omega)$ with equal mass, we define
\begin{equation*}
u(\mu_+,\mu_-):V\setminus\Omega\rightarrow\R, \ x\mapsto (\mu_+-\mu_-)\left(\partial\Omega_{\gamma^{-1}(\pi_{\partial\Omega}(x))}\right)
\end{equation*}
and the sets of \emph{admissible fluxes and images} as 
\begin{align*}
\A_\flux &= \{ \flux\in\M(\overline{\Omega};\R^2) \,|\, \flux\text{ is mass flux between }\mu_+\text{ and }\mu_- \}, \\
\A_u &=  \{ u\in \BV(V) \,|\, u=u(\mu_+,\mu_-) \text{ on } V\setminus\overline{\Omega} \}.
\end{align*} 
\end{definition}

\begin{figure}
\subfloat[]{%
\setlength\unitlength\textwidth
\begin{picture}(.6,.3)
\put(.001,.01){\color{black!70}\rule{.06\unitlength}{.3\unitlength}}
\put(.061,.01){\color{black!50}\rule{.06\unitlength}{.3\unitlength}}
\put(.121,.01){\color{black!30}\rule{.06\unitlength}{.032\unitlength}}
\put(.121,.278){\color{black!30}\rule{.06\unitlength}{.032\unitlength}}
\put(.181,.01){\color{black!10}\rule{.06\unitlength}{.032\unitlength}}
\put(.181,.278){\color{black!10}\rule{.06\unitlength}{.032\unitlength}}
\put(0,0){\includegraphics[width=0.6\unitlength]{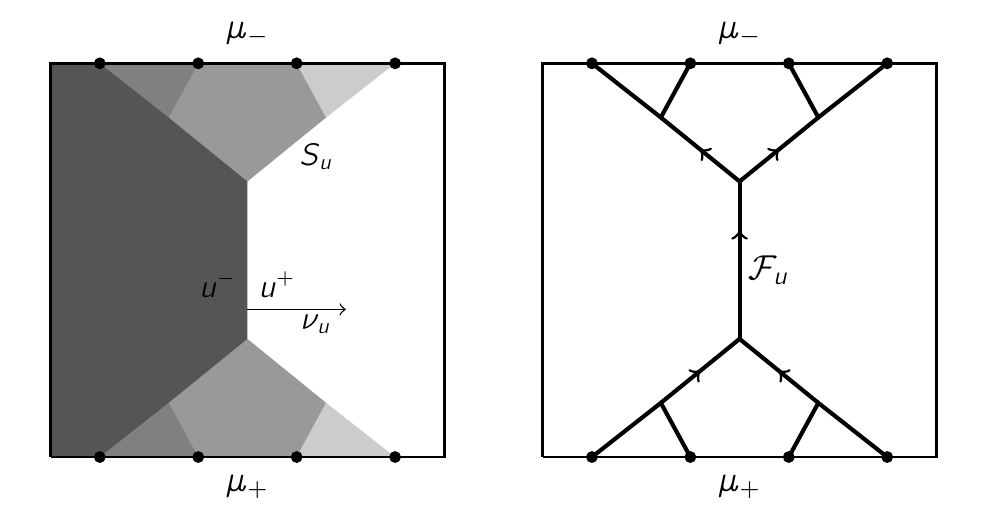}}
\put(.24,.24){\small$\Omega$}
\put(.54,.24){\small$\Omega$}
\end{picture}}
\subfloat[]{\includegraphics[width=0.39\textwidth]{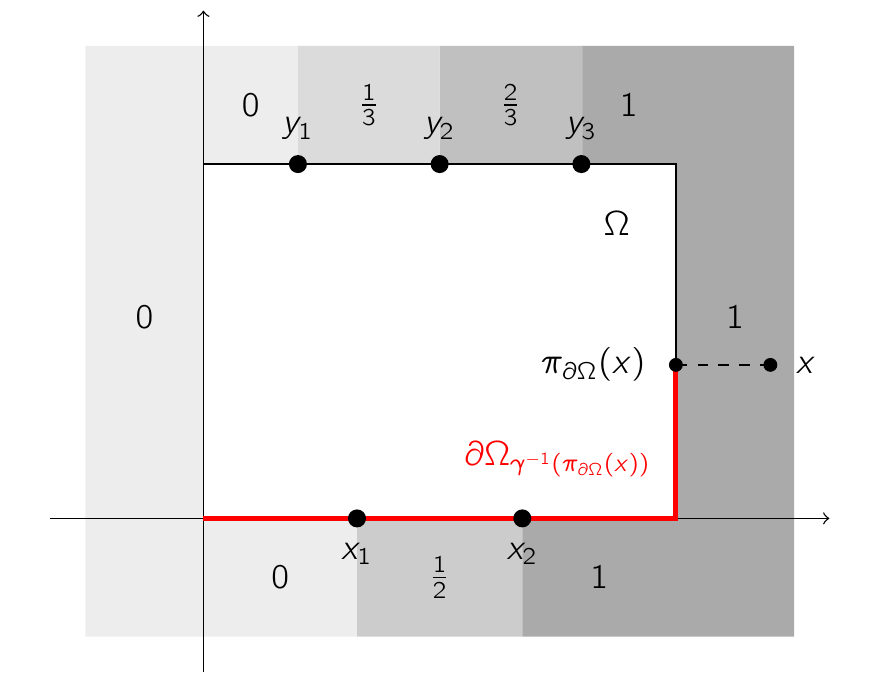}}
\caption{(a) A grey value image $u$ and its corresponding mass flux $\flux_u$. (b) Sketch of $u\left(\mu_+,\mu_-\right)$ for $\mu_+=\frac{1}{2}(\delta_{x_1}+\delta_{x_2})$ and $\mu_-=\frac{1}{3}(\delta_{y_1}+\delta_{y_2}+\delta_{y_3})$ so that $u\left(\mu_+,\mu_-\right)$ takes values in $\{0,\frac{1}{3},\frac{1}{2},\frac{2}{3},1\}$.}
\label{fig:ComparisonFluxImage}
\end{figure}
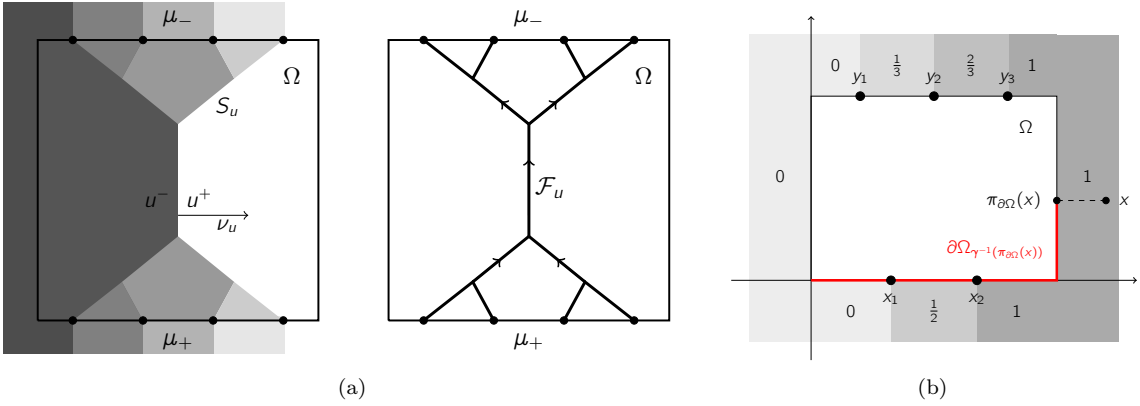

By \cite[Lem.\,3.1.3]{BrRoWi2018} the mapping $u\mapsto\flux_u\restr\overline{\Omega}$ from $\A_u$ to $\A_\flux$ is one-to-one
so that we may also introduce the image $u_\flux\in\A_u$ corresponding to the mass flux $\flux\in\A_\flux$.
The relation between images and fluxes is illustrated in \cref{fig:ComparisonFluxImage}.
The following cost functional now expresses the generalized branched transport cost as a cost of images.

\begin{definition}[Image-based cost functional]\label{def:ImageBasedCostFunctionals}
For an admissible image $u\in\A_u$, the \emph{generalized branched transport cost of images} is defined as 
\begin{align*}
\tilde{\E}(u) = \int_{S_u\cap\overline{\Omega}} \tau(|[u]|) \d\hd^1 + \tau'(0)|Du|(\overline{\Omega}\setminus S_u),
\end{align*}
where $\tau'(0)\in(0,\infty]$ denotes the right derivative of $\tau$ in $0$.
\end{definition}

In \cite[Thm.\,3.2.2 \& Lem.\,3.2.5]{BrRoWi2018} we proved the relation $\E(\flux) \geq \tilde{\E}(u_\flux)$
by showing that both functionals coincide for discrete mass fluxes and the corresponding images
and by then exploiting that $\tilde\E$ is lower semi-continuous while $\E$ is the relaxation of its restriction to discrete mass fluxes
(that is, the largest lower semi-continuous function which coincides with $\tilde\E$ on discrete mass fluxes).
The opposite inequality can be obtained by showing that $\tilde\E$ is a relaxation as well,
an issue which was considered in \cite{BrWi2018,MaWi19}.

\begin{theorem}[Equality of flux-based and image-based cost]\label{thm:fluxImage}
Let $\mu_+,\mu_-\in\M_+(\partial\Omega)$ with equal mass. For a mass flux $\flux\in\A_\flux$ and the corresponding image $u_\flux\in\A_u$ we have $\E(\flux) = \tilde{\E}(u_\flux)$. 
\end{theorem}

\begin{proof}
Since the relation between $\A_u$ and $\A_\flux$ is one-to-one it suffices to show $\E(\flux_u) = \tilde{\E}(u)$ for any $u\in\A_u$.
Now note that $\flux_u=\theta(\hd^1\restr S)+\flux^d$ for $S=S_u\cap\overline\Omega$, $\theta=[u]\nu_u^\perp$, and $\flux^d=(\nabla u^\perp \leb^2\restr V + D_cu^\perp)\restr\overline\Omega$.
By \cite[Thm.\,3.78]{AmFuPa2000}, $\Sigma$ is countably $\hd^1$-rectifiable,
and by \cite[Lem.\,3.76]{AmFuPa2000} $\flux^d$ is diffuse, that is, singular with respect to $\hd^1\restr R$ for any countably $1$-rectifiable $R\subset\overline\Omega$.
Thus, by \cite[Prop.\,2.32]{BrWi2018} we have
\begin{equation*}
\E(\flux) = \int_S \tau(|\theta|) \, \d\hd^1 + \tau'(0)|\flux^d|(\overline{\Omega}),
\end{equation*}
however, this equals exactly $\tilde\E(u)$.
\end{proof} 

It turns out that $\tilde\E$ can be expressed as an energy of a surface in $\R^3$, which will lead to a convex optimiziation problem.
This approach has been introduced in \cite{AlBoDa03} to prove optimality of special solutions to the Mumford--Shah problem (and related ones)
by exhibiting a lower bound on the surface energy,
and it was subsequently exploited in \cite{PoCrBiCh2009,PoCrBiCh2010} to numerically compute global minimizers of the Mumford--Shah functional.
The setting in \cite{AlBoDa03,PoCrBiCh2009,PoCrBiCh2010} is slightly more general than what we need here.
The authors consider a generalized Mumford--Shah functional
\begin{equation*}
J(u) = \int_V g(x,u(x),\nabla u(x)) \,\d x + \int_{S_u} h(x,u^+,u^-,\nu_u)\,\d\hd^1(x),
\end{equation*} 
where $g$ is a normal Caratheodory function convex in its third argument
and $h$ is one-homogeneous and convex in its last argument and subadditive in $(u^+,u^-)$ (see \cite[\S\,5.2-5.3]{AmFuPa2000} for details on the requirements).
In \cite{AlBoDa03,PoCrBiCh2010} it is shown that $J(u)$ can be estimated from below as follows. Let
\begin{equation*} 
1_u:V\times\R\to\{0,1\}, \quad 1_u(x,s)=\begin{cases} 1 & \text{ if }u(x)>s, \\ 0 & \text{ otherwise} \end{cases}
\end{equation*}
denote the characteristic function of the subgraph of the image $u\in\BV(V)$ and introduce the convex set
\begin{multline*}
\K = \Big\{ \phi=(\phi^x,\phi^s)\in C_0^\infty(V\times\R;\R^2\times\R) \,\Big|\, \phi^s(x,s)\geq g^*(x,s,\phi^x(x,s)) \ \forall \ (x,s)\in V\times\R, \\\textstyle
\left|\int_{s_1}^{s_2} \phi^x(x,s)\d s\right| \leq h(x,s_1,s_2,\nu) \ \forall \ x\in V, s_1<s_2, \nu\in S^1 \Big\}
\end{multline*}
of three-dimensional vector fields, where $g^*$ denotes the Legendre--Fenchel conjugate of $g$ with respect to its last argument.
Then the generalized Mumford--Shah functional can be estimated via
\begin{equation*}
J(u) \geq \underset{\phi\in\K}{\sup} \int_{V\times\R}\phi\cdot \d D1_u,
\end{equation*}
where the right-hand integral can be interpreted as an integral over the complete graph of $u$ and thus as a surface functional.
Even equality is expected, but has not been rigorously proved.
The above can be specialized to our setting by picking $g(x,u,p)=\tau'(0)|p|$ and $\psi(x,u^+,u^+,\nu)=\tau(|u^+-u^-|)$.

\begin{definition}[Surface-based cost functional]\label{def:surfaceCost}
Let $\mu_+,\mu_-\in\M_+(\partial\Omega)$ with equal mass.
We set
\begin{multline*}
\K = \Big\{ \phi=(\phi^x,\phi^s)\in C_0^\infty(V\times\R;\R^2\times\R) \,\Big|\, |\phi^x(x,s)|\leq\tau'(0),\,\phi^s(x,s)\geq0 \ \forall \ (x,s)\in V\times\R, \\
\textstyle\qquad\hfill\left|\int_{s_1}^{s_2} \phi^x(x,s)\d s\right| \leq \tau(s_2-s_1) \ \forall \ x\in V, s_1<s_2 \Big\}.
\end{multline*}
For an admissible image $u\in\A_u$ the \emph{generalized branched transport cost of surfaces} is defined as
\begin{equation*}
\G(1_u) = \sup_{\phi\in\K}\int_{\overline\Omega\times\R}\phi\cdot\d D1_u.
\end{equation*}
\end{definition}

In \cite{AlBoDa03,BrRoWi2018} it is shown that $\tilde\E(u)\geq\G(1_u)$; we now show equality.

\begin{theorem}[Equality of image-based and surface-based cost]\label{thm:equivalenceImageSurface}
Let $\mu_+,\mu_-\in\M_+(\partial\Omega)$ with equal mass and assume without loss of generality that $u(\mu_+,\mu_-)$ takes minimum value $0$ and maximum value $M\leq\|\mu_+\|_\M$.
For an image $u\in\A_u$ we have $\tilde{\E}(u)=\G(1_u)$. Moreover,
\begin{equation*}
\min_{u\in\A_u}\G(1_u)
=\min_{u\in\A_u\cap\BV(V;[0,M])}\G(1_u)
=\min_{u\in\A_u\cap\BV(V;[0,M])}\sup_{\phi\in\tilde\K}\int_{\overline\Omega\times[0,M]}\phi\cdot\d D1_u
\end{equation*}
for the set
\begin{multline*}
\tilde\K = \Big\{ \phi=(\phi^x,\phi^s)\in C^0(\overline\Omega\times[0,M];\R^2\times[0,\infty)) \,\Big|\, |\phi^x(x,s)|\leq\tau'(0),\,\phi^s(x,s)\geq0\ \forall \ (x,s)\in \overline\Omega\times[0,M], \\
\textstyle\qquad\hfill\left|\int_{s_1}^{s_2} \phi^x(x,s)\d s\right| \leq \tau(s_2-s_1) \ \forall \ x\in\overline\Omega, 0\leq s_1<s_2\leq M \Big\}.
\end{multline*}
\end{theorem}

\begin{proof}
We need to show $\tilde\E(u)\leq\G(1_u)$.
To this end it suffices to consider $\tau$ with $\tau(m)=\alpha m$ for all $m\leq m_0$, where $\alpha<\infty$ and $m_0>0$ are arbitrary.
Indeed, assume equality for such transportation costs, let $\tau$ be a given transportation cost, and set $\tau_n(m)=\min\{\tau(m),\alpha_nm\}$
for $0<\alpha_1<\alpha_2<\ldots$ a sequence with $\alpha_n\to\tau'(0)$ as $n\to\infty$.
Decorating $\G$ and $\tilde\E$ with a superscript to indicate what transportation cost they are based on, we have
\begin{equation*}
\G^\tau(1_u)\geq\G^{\tau_n}(1_u)=\tilde\E^{\tau_n}(u)\to\tilde\E^\tau(u)\quad\text{as }n\to\infty
\end{equation*}
by monotone convergence, as desired.

By \cite[Thm.\,3.78]{AmFuPa2000}, $S_u$ is countably $\hd^1$-rectifiable.
Furthermore, $[u]\in L^1(\hd^1\restr S_u)$.
Thus, the jump part of $Du$ can be treated via a decomposition strategy as for instance also used in \cite[Lem.\,4.2]{ChMeFe2019}.
In detail, let $\varepsilon>0$ be arbitrary.
Since $S_u$ is rectifiable there exists a compact (oriented) $C^1$-manifold $\manifold\subset\R^2$ with $|Du|(S_u\setminus\manifold)<\varepsilon$.
For every $x\in\manifold$ and $\ell>0$ let us denote by $F_x^\ell\subset\R^2$ the closed square of side length $2\ell$,
centred at $x$ and axis-aligned with the tangent and the normal vector to $\manifold$ in $x$.
Also denote that rigid motion by $R_x:\R^2\to\R^2$ which maps $x$ to $0$ and the unit tangent of $\manifold$ in $x$ to $(0,1)$ (thus $R_x(F_x^\ell)=[-\ell,\ell]^2$).
Now fix $\delta>0$ such that we have
\begin{gather*}
R_x(\manifold\cap F_x^\delta)\text{ is the graph of a map }g_x\in C^1([-\delta,\delta];[-\delta,\delta])\text{ with }g_x(0)=g_x'(0)=0,\\
|Du|\left(\bigcup_{x\in\manifold}F_x^\delta\setminus S_u\right)<\varepsilon
\end{gather*}
(the latter can be achieved since $\bigcup_{x\in\manifold}F_x^\delta\setminus S_u\to\manifold\setminus S_u$ monotonically as $\delta\to0$
and thus by outer regularity of $|Du|$ we have $|Du|(\bigcup_{x\in\manifold}F_x^\delta\setminus S_u)\to|Du|(\manifold\setminus S_u)=0$ as $\delta\to0$).
Now $\manifold\subset\bigcup_{x\in\manifold}\bigcup_{\ell<\delta}F_x^\ell$
so that by Vitali--Besicovitch covering theorem \cite[Thm.\,2.19]{AmFuPa2000}
there is a countable disjoint family of cubes $F_{x_1}^{\ell_1},F_{x_2}^{\ell_2},\ldots$, $x_i\in\manifold$, $\ell_i<\delta$,
whose union $F$ satisfies $|Du|(\manifold\setminus F)=0$.
By taking a finite subfamily $F_{x_1}^{\ell_1},\ldots,F_{x_K}^{\ell_K}$ we achieve $|Du|(\manifold\setminus\bigcup_{k=1}^KF_{x_k}^{\ell_k})<\varepsilon$.

On $F_{x_i}^{\ell_i}$ define the projection $p_i:F_{x_i}^{\ell_i}\to\manifold$ by
\begin{equation*}
p_i=R_{x_i}^{-1}\circ g_{x_i}\circ\pi_{\R\times\{0\}}\circ R_{x_i}
\end{equation*}
(where $\pi_{\R\times\{0\}}$ returns first coordinate of a vector; $p_i$ is the projection along one axis direction of $F_{x_i}^{\ell_i}$).
Furthermore define $\psi$ as
\begin{equation*}
\tilde\psi_i(x,s)=\begin{cases}
\nu_{\manifold}(p_i(x))\frac{\tau(|[u](p_i(x))|)}{[u](p_i(x))}&\text{if }\dist(x,\partial F_{x_i}^{\ell_i})>\eta\text{ and}\\&\quad s\in[\min\{u^-(p_i(x)),u^+(p_i(x))\},\max\{u^-(p_i(x)),u^+(p_i(x))\}],\\
0&\text{else}
\end{cases}
\end{equation*}
with $\nu_{\manifold}$ being the normal vector to $\manifold$.
Note that by construction we have $\left|\int_{s_1}^{s_2}\tilde\psi_i(x,s)\,\d s\right|\leq\tau(s_2-s_1)$ for all $s_1<s_2$, $x\in F_{x_i}^{\ell_i}$
as well as $|\tilde\psi_i(x,s)|\leq\alpha$ due to $\tau(m)\leq\alpha m$.
Mollifying $\tilde\psi_i$ with a mollifier $\rho_\eta(x)=\rho(x/\eta)/\eta$ for $\rho\in C_0^\infty([-1,1]^3;[0,\infty))$ with unit integral,
the above constraints stay satisfied by Jensen's inequality,
and we obtain some $\psi_i=\rho_\eta*\tilde\psi_i\in C_0^\infty(F_{x_i}^{\ell_i}\times\R;\R^2)$.
Extending $\psi_i$ by zero to $V\times\R$ we can define $\phi_i\in\K$ as $\phi_i(x,s)=(\psi_i(x,s),0)$.
We now set $\hat\phi=\sum_{i=1}^K\phi_i$.
Note that in the above we can choose $\eta$ small enough such that $\int_{\overline\Omega\times\R}\hat\phi\cdot\d D1_u\geq\int_{S_u}\tau(|[u]|)\,\d\hd^1-5\alpha\varepsilon$.
Indeed, abbreviating
\begin{equation*}
\chi_{(a,b)}(s)=\begin{cases}1&\text{if }a<s<b,\\-1&\text{if }b<s<a,\\0&\text{else,}\end{cases}
\end{equation*}
we can calculate
\begin{align*}
\int_{\overline\Omega\times\R}\hat\phi\cdot\d D1_u
&=\int_{(\overline\Omega\cap\bigcup_{i=1}^KF_{x_i}^{\ell_i})\times\R}\hat\phi\cdot\d D1_u
\geq\int_{(\overline\Omega\cap S_u\cap\bigcup_{i=1}^KF_{x_i}^{\ell_i})\times\R}\hat\phi\cdot\d D1_u-\alpha\varepsilon\\
&=\sum_{i=1}^K\int_{\overline\Omega\cap\manifold\cap F_{x_i}^{\ell_i}}\int_\R\chi_{(u^-(x),u^+(x))}(s)\nu_\manifold(x)\cdot\psi_i(x,s)\,\d s\,\d\hd^1(x)-2\alpha\varepsilon\\
&\underset{\eta\to0}{\longrightarrow}\sum_{i=1}^K\int_{\overline\Omega\cap\manifold\cap F_{x_i}^{\ell_i}}\int_\R\chi_{(u^-(x),u^+(x))}(s)\nu_\manifold(x)\cdot\tilde\psi_i(x,s)\,\d s\,\d\hd^1(x)-2\alpha\varepsilon\\
&=\sum_{i=1}^K\int_{\overline\Omega\cap\manifold\cap F_{x_i}^{\ell_i}}\tau(|[u](x)|)\,\d\hd^1(x)-2\alpha\varepsilon
\geq\int_{\overline\Omega\cap\manifold}\tau(|[u](x)|)\,\d\hd^1(x)-3\alpha\varepsilon\\
&\geq\int_{\overline\Omega\cap S_u}\tau(|[u](x)|)\,\d\hd^1(x)-4\alpha\varepsilon
\end{align*}
since for $\eta\to0$ the $\psi_i$ converge to the $\tilde\psi_i$ in $L^1(\hd^2\restr\manifold\times\R)$
and $(x,s)\mapsto\chi_{(u^-(x),u^+(x))}(s)\nu_\manifold(x)$ is in $L^\infty(\hd^2\restr\manifold\times\R)$.

Now consider the cost associated with the diffuse part of $Du$.
By \cite[Prop.\,3.64]{AmFuPa2000}, $u_\xi=\rho_\xi*u\to\tilde u$ pointwise on $V\setminus S_u$,
where $\rho_\xi$ is some mollifier with length scale $\xi$ and $\tilde u$ is a particular representative of $u$, the so-called approximate limit.
Consequently, $u_\xi\to\tilde u$ pointwise $|Du|\restr V\setminus S_u$-almost everyhere.
Thus, by Egorov's theorem there exists some measurable set $B\subset V$ such that $|Du|(V\setminus S_u\setminus B)<\varepsilon$ and $u_\xi\to\tilde u$ uniformly on $B$.
Let $\xi$ be small enough such that $|\tilde u-u_\xi|<m_0/4$ on $B$
and let $\psi\in C_0^\infty(V;\R^2)$ such that $|\psi|\leq1$ everywhere and $\int_{\overline\Omega}\psi\cdot\d Du\geq|Du|(\overline\Omega)-\varepsilon$.
Furthermore fix $\eta>0$ such that $|Du|\left(U_\eta\setminus\bigcup_{k=1}^KF_{x_k}^{\ell_k}\right)$ for the $\eta$-neighbourhood $U_\eta$ of $\partial V\cup\bigcup_{k=1}^KF_{x_k}^{\ell_k}$.
We now define
\begin{equation*}
\bar\phi(x,s)={\alpha\psi(x)\chi_1(x)\chi_2(t-u_\xi(x))\choose0},
\end{equation*}
where $\chi_1\in C_0^\infty(V;[0,1])$ is a cutoff function which is zero on $\bigcup_{k=1}^KF_{x_k}^{\ell_k}$ and one outside $U_\eta$,
and where $\chi_2\in C_0^\infty(\R;[0,1])$ is a cutoff function which is one on $[-m_0/4,m_0/4]$ and zero outside $[-m_0/2,m_0/2]$.
Note that $\bar\phi\in\K$ by construction and
\begin{align*}
\int_{\overline\Omega\times\R}\bar\phi\cdot\d D1_u
&=\alpha\int_{\{(x,\tilde u(x))\,|\,x\in\overline\Omega\}}\chi_1(x){\psi(x)\choose0}\cdot\d D1_u(x,s)\\
&\geq\alpha\int_{\{(x,\tilde u(x))\,|\,x\in\overline\Omega\setminus S_u\}}{\psi\choose0}\cdot\d D1_u
-\alpha|Du|(S_u\setminus U_\eta)-\alpha|Du|\left(U_\eta\setminus\bigcup_{k=1}^KF_{x_k}^{\ell_k}\right)\\&\qquad-\alpha|Du|(\overline\Omega\setminus S_u\setminus B)\\
&\geq\alpha\int_{\{(x,\tilde u(x))\,|\,x\in\overline\Omega\setminus S_u\}}{\psi\choose0}\cdot\d D1_u
-3\alpha\varepsilon
=\alpha\int_\R\int_{\overline\Omega\setminus S_u}\psi\cdot\d D\chi_{\{u>s\}}\,\d s
-3\alpha\varepsilon\\
&=\alpha\int_{\overline\Omega\setminus S_u}\psi\cdot\d Du
-3\alpha\varepsilon
\geq\alpha|Du|(\overline\Omega\setminus S_u)
-4\alpha\varepsilon,
\end{align*}
where $\chi_{\{u>s\}}$ is the characteristic function of the $s$-superlevel set of $u$
and where in the last equality we used the coarea formula.

Summarizing, we have $\phi=\hat\phi+\bar\phi\in\K$ with
$\int_{\overline\Omega\times\R}\phi\cdot\d D1_u\geq\int_{\overline\Omega\cap S_u}\tau(|[u](x)|)\,\d\hd^1(x)+\alpha|Du|(\overline\Omega\setminus S_u)-8\alpha\varepsilon=\tilde\E(u)-8\alpha\varepsilon$,
and thus $\G(1_u)\geq\tilde\E(u)$ follows from the arbitrariness of $\varepsilon$.

From the definition of $\tilde\E$ it is obvious that $\tilde\E(u)$ decreases if $u$ is clipped to the range $[0,M]$.
Thus, minimizers of $\G(1_u)=\tilde\E(u)$ among all admissible images $u$ lie in $\A_u\cap\BV(V;[0,M])$,
and one may restrict the integral in the definition of $\G$ to $\overline\Omega\times[0,M]$.
Finally, by density of $\{\phi\in\K\,|\,\phi^s=0\}$ in $\tilde\K$ with respect to the supremum norm, we may replace $\K$ with $\tilde\K$ without changing the supremum.
\end{proof}

Note that we could even set $\phi^s\equiv0$ in $\tilde\K$ without changing $\sup_{\phi\in\tilde\K}\int_{\overline\Omega\times[0,M]}\phi\cdot\d D1_u$
since the integral increases if $\phi^s$ decreases.
The problem of minimizing $\G(1_u)$ among all characteristic functions of subgraphs of admissible images is not convex, since the space of characteristic functions is not.
The underlying idea of \cite{AlBoDa03,PoCrBiCh2009,BrRoWi2018} is that one does not lose much by convexifying the domain of $\G$ as follows.

\begin{definition}[Convex cost functional]\label{def:convexCost}
Let $\mu_+,\mu_-\in\M_+(\partial\Omega)$ with equal mass and $M$, $\tilde\K$ from \cref{thm:equivalenceImageSurface}. We set
\begin{equation*}
\C = \{ v\in \BV(V\times\R;[0,1]) \,|\, v=1_{u(\mu_+,\mu_-)} \text{ on } (V\times\R)\setminus(\overline{\Omega}\times[0,M]) \},
\end{equation*}
where we extended $1_{u(\mu_+,\mu_-)}$ by $1$ to $V\times(-\infty,0)$ and by $0$ to $V\times(M,\infty)$.
The \emph{convex generalized branched transport cost} is $\tilde\G:\C\to\R$,
\begin{equation*}
\tilde\G(v)=\sup_{\phi\in\tilde\K}\int_{\overline\Omega\times[0,M]}\phi\cdot\d Dv.
\end{equation*}
\end{definition}

By definition and \cref{thm:equivalenceImageSurface}, $\tilde\G$ coincides with $\G$ on functions of the form $v=1_u$ with $u\in\A_u$.
The following \namecref{thm:propertiesConvex} shows that the problem of minimizing $\tilde\G$
is related to the original generalized branched transport problem in the sense that if the minimizer of $\tilde\G$ is binary, then it is a solution of the original problem.
The \namecref{thm:propertiesConvex} also shows that the original and the convex minimization problem cannot be fully equivalent since sometimes $\tilde\G$ has nonbinary minimizers
(however, those non-binary minimizers may coexist with binary minimizers so that the minimization problems might still be equivalent after selecting the binary minimizers).

\begin{proposition}[Properties of convex cost functional]\label{thm:propertiesConvex}
Let $\mu_+,\mu_-\in\M_+(\partial\Omega)$ with equal mass.
\begin{enumerate}
\item
$\tilde\G$ is convex, weakly-* lower semi-continuous, and satisfies $\inf_{v\in\C}\tilde\G(v)\leq\min_{u\in\A_u}\G(1_u)$.
\item
If a minimizer $v\in\C$ of $\tilde\G$ is binary, then $v=1_u$ for a minimizer $u\in\A_u$ of $\G(1_u)$.
\item\label{item:nonbinary}
If $\tau$ is not linear, there exist $\mu_+,\mu_-\in\M_+(\partial\Omega)$ such that if $\tilde\G$ has minimizers, at least some of them are nonbinary.
\end{enumerate}
\end{proposition}
\begin{proof}
\begin{enumerate}
\item
As the supremum over linear functionals on a convex domain, $\tilde\G$ is convex and lower semi-continuous with respect to the weak-* topology.
Furthermore, $\inf_v\tilde\G(v)\leq\inf_{u\in\A_u}\tilde\G(1_u)=\min_{u\in\A_u}\G(1_u)$.
\item
First note that $\tilde\G(v)=\infty$ unless $v$ is monotonically decreasing in $s$-direction.
Indeed, if $(Dv)_3$ is not nonpositive, there exists a continuous $\phi^s\geq0$ with $\int_{\overline\Omega\times[0,M]}\phi^s\,\d(Dv)_3>0$
(for instance take the positive part of some $\psi\in C^0(\overline\Omega\times[0,M])$ with $\int_{\overline\Omega\times[0,M]}\psi\,\d(Dv)_3\approx\|(Dv)_3\|_\M$)
so that $\tilde\G(v)\geq\sup_{\lambda>0}\int_{\overline\Omega\times[0,M]}(0,0,\lambda\phi^s)\cdot\d Dv=\infty$.
Thus, $v$ can be represented as $1_u$ for some function $u\in\A_u$.
Due to the previous point, $1_u$ must be a minimizer of $\G$.
\item
Assume the contrary, that is, for any $\mu_+,\mu_-\in\M_+(\partial\Omega)$ with equal mass the minimizers of $\tilde\G$ are binary.
Since $\tau$ is not linear, there exist $\mu_+,\mu_-$ such that the corresponding generalized branched transport problem has no unique minimizer
(see for instance \cref{fig:Example2To2BTConvexification}).
Thus, there are $u_1,u_2\in\A_u$, $u_1\neq u_2$ with $\min_{u\in\A_u}\G(1_u)=\G(1_{u_1})=\G(1_{u_2})=\tilde\G(1_{u_1})=\tilde\G(1_{u_2})=\min_{v\in\C}\tilde\G(v)$,
where the last equality follows from the previous point.
However, since $\C$ and $\tilde\G$ are convex, $(1_{u_1}+1_{u_2})/2$ is also a minimizer of $\tilde\G$, which is nonbinary.
\qedhere
\end{enumerate}
\end{proof}

\begin{figure}
	\centering
	\includegraphics[width=0.6\textwidth]{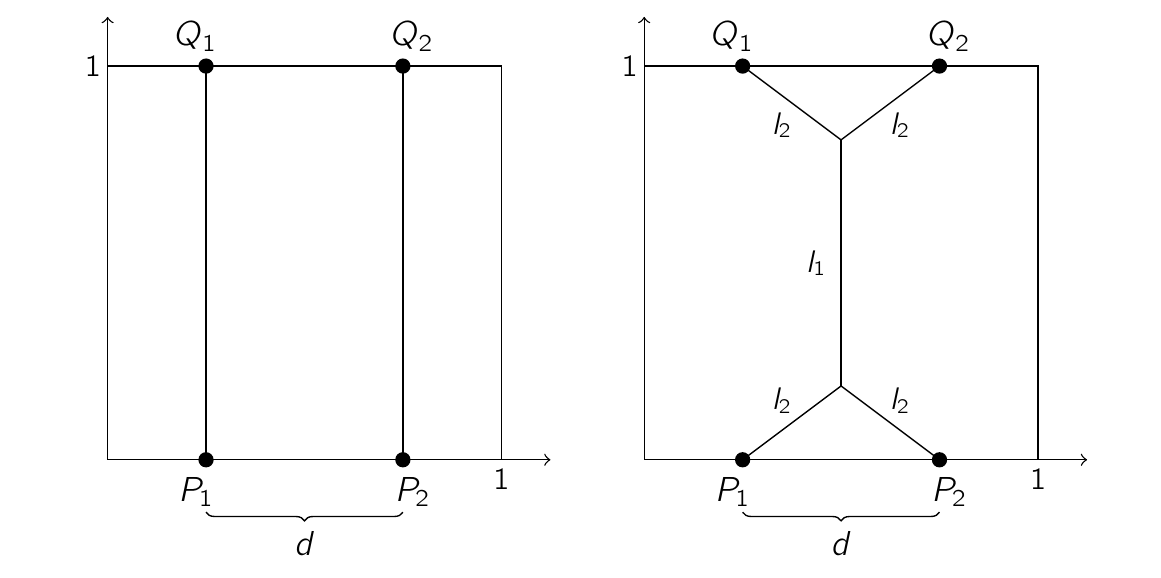}
	\caption{Generalized branched transport in $\Omega=[0,1]^2$ from two points at $P_1,P_2\in\partial\Omega$ with equal mass to two points $Q_1,Q_1\in\partial\Omega$ with equal mass.
	Depending on the mass and the point distance $d$, the optimal network has the left or the right topology.
	At the bifurcation point, both topologies are optimal.}
	\label{fig:Example2To2BTConvexification}
\end{figure}

When surface energies are relaxed to energies over functions $v\in\BV(V\times\R;[0,1])$ as in our case, one typically uses the coarea formula to show
that for a minimizer $v$ the characteristic functions of its superlevel sets have the same minimizing cost and thus there are always binary minimizers.
In the case of a one-homogeneous $\tau(m)=\alpha m$ this works as follows,
\begin{equation*}
\tilde\G(v)
=\int_{\overline\Omega\times[0,M]}\tau\left(\frac{Dv}{|Dv|}\right)\,\d Dv
=\int_0^1\int_{\overline\Omega\times[0,M]}\tau\left(\frac{D\chi_{\{v>t\}}}{|D\chi_{\{v>t\}}|}\right)\,\d D\chi_{\{v>t\}}\,\d t
=\int_0^1\tilde\G(\chi_{\{v>t\}})\,\d t,
\end{equation*}
where we exploited the one-homogeneity of $\tau$ and the coarea formula
(a similar calculation can be performed for the lifting of the generalized Mumford--Shah functional $J$ with $h\equiv0$).
If $v$ is a minimizer so that $\tilde\G(v)\leq\tilde\G(\chi_{\{v>t\}})$ for all $t$, then by the above equality we necessarily have $\tilde\G(v)=\tilde\G(\chi_{\{v>t\}})$ for almost all $t\in[0,1]$.
However, a formula as the above is not true in our case.

\begin{proposition}[Convex cost of superlevel sets]
It holds $\tilde\G(v)\leq\int_0^1\tilde\G(\chi_{\{v>t\}})\,\d t$, and this is not an equality.
\end{proposition}
\begin{proof}
The inequality holds by the convexity of $\tilde\G$ and Jensen's inequality in combination with $v=\int_0^1\chi_{\{v>t\}}\,\d t$.
To show that the inequality is sometimes strict,
first note that by an analogous construction as in the proof of \cref{thm:equivalenceImageSurface}
we have $\tilde\G(v)=\tilde\G_1(v)+\tilde\G_2(v)=\sup_{\phi\in\tilde\K}\int_{\partial\Omega\times\R}\phi\cdot\d Dv+\sup_{\phi\in\tilde\K}\int_{\Omega\times\R}\phi\cdot\d Dv$
with $\tilde\G_i(v)\leq\int_0^1\tilde\G_i(\chi_{\{v>t\}})\,\d t$, $i=1,2$, for the same reason as above.
Now consider the example $v(x,s)=\int_0^11_{u+tc}(x,s)\,\d t$ for $(x,s)\in\Omega\times\R$ with $u(x)=mx_1$ for some $c,m>0$.
We have
\begin{equation*}
\tilde\G_2(\chi_{\{v>t\}})
=\G(1_{u+tc})
=\tilde\E(u)
=\tau'(0)m\leb^2(\Omega),
\end{equation*}
while
\begin{align*}\textstyle
\tilde\G_2(v)
&\leq\int_\Omega\sup_{\phi\in\tilde\K}\int_0^M\phi(x,s)\cdot\nabla v(x,s)\,\d s\,\d x
=\int_\Omega\sup_{\phi\in\tilde\K}\int_{u(x)-c}^{u(x)}\phi_1(x,s)\tfrac{m}c\,\d s\,\d x\\
&=\leb^2(\Omega)\tfrac mc\sup\left\{\int_0^c\psi\,\d s\,\middle|\,\psi:[0,c]\to\R,\left|\int_{s_1}^{s_2}\psi(s)\,\d s\right|\leq\tau(s_2-s_1)\text{ for all }0\leq s_1<s_2\leq c\right\}\\
&=\leb^2(\Omega)m\tfrac{\tau(c)}c.
\end{align*}
Summarizing, $\tilde\G_2(v)\leq\leb^2(\Omega)m\tfrac{\tau(c)}c<\tau'(0)m\leb^2(\Omega)=\int_0^1\tilde\G_2(\chi_{\{v>t\}})\,\d t$, as desired.
\notinclude{Note that one can easily check by checking the Euler--Lagrange equations that at each $x$ with $\{x\}\times\R\cap S_v\neq\emptyset$,
$\psi(x,s)=\tau'(s-r)\nu_{S_v}(x,s)$ if $[v](x,s)$ is decreasing in $s$ starting from $r$.}
\end{proof}

This does not imply that $\tilde\G$ does not always have binary minimizers;
intuitively, while nonbinary functions $v$ may have smaller costs in the domain interior,
one has to pay some extra cost for the transition from binary on $\partial\Omega\times[0,M]$ to nonbinary in $\Omega\times[0,M]$.
The next section and the numerical experiments provide evidence that binary minimizers exist at least in many relevant cases,
as is also believed for the generalized Mumford--Shah setting.

The last remaining inequality in \eqref{eqn:mainInequalities} bounds the convex saddle point problem $\inf_{v\in\C}\sup_{\phi\in\tilde\K}\int_{\overline\Omega\times[0,M]}\phi\cdot\d Dv$ below
by the corresponding primal optimization problem in the vector field $\phi$ (note that in \eqref{eqn:mainInequalities} we did not reduce $\K$ to $\tilde\K$ for simplicity of exposition).
To have an equality we thus need to show strong duality.

\begin{theorem}[Strong duality for convex cost]\label{thm:strongDuality}
	Let $\mu_+,\mu_-,\C,M,\tilde\K$ as in \cref{def:convexCost}. $\tilde\G$ has a minimizer, and we have the strong duality
	\begin{gather*}
	\min_{v\in\C}\tilde\G(v)
	=\min_{v\in\C} \sup_{\phi\in\tilde\K} \int_{\overline\Omega\times[0,M]} \phi\cdot\d Dv
	=\sup_{\phi\in\tilde\K} \inf_{v\in\C} \int_{\overline\Omega\times[0,M]} \phi\cdot\d Dv
	= \sup_{\phi\in\tilde\K\cap C^1(\overline\Omega\times[0,M];\R^2\times\R)}\D(\phi)\\
	\text{for }\D(\phi)=\int_{\partial(\Omega\times(0,M))}1_{u(\mu_+,\mu_-)}\phi\cdot n\,\d\hd^2-\int_{\Omega\times(0,M)}\max\{0,\dive\phi\}\,\d x\,\d s.
	\end{gather*}
\end{theorem}
\begin{proof}
The last equality is obtained via the integration by parts
\begin{equation*}
\int_{\overline\Omega\times[0,M]} \phi\cdot\d Dv
=\int_{\partial(\Omega\times(0,M))}v\phi\cdot n\,\d\hd^2-\int_{\Omega\times(0,M)}v\dive\phi\,\d x\,\d s,
\end{equation*}
noticing $v=1_{u(\mu_+,\mu_-)}$ on $\partial\Omega\times[0,M]$ and taking in $\Omega\times(0,M)$ the maximizing $v=1$ if $\dive\phi>0$ and $0$ else
(we also exploited denseness of $C^1$ in $C^0$).
As for the first equality, the strong duality, we define $\XX=C^1(\overline\Omega\times[0,M];\R^2\times\R)$, $\Y=C^0(\overline\Omega\times[0,M])\times C^0(\overline\Omega\times[0,M];\R^2\times\R)$, as well as
\begin{align*}
A:&\XX\to\Y,&A\phi&=(\dive\phi,\phi),\\
F:&\XX\to[0,\infty],&F(\phi)&=\iota_{\tilde\K}(\phi),\\
G:&\Y\to\R,&G(\psi_1,\psi_2)&=\int_{\Omega\times(0,M)}\max\{0,\psi_1\}\,\d x\,\d s-\int_{\partial(\Omega\times(0,M))}1_{u(\mu_+,\mu_-)}\psi_2\cdot n\,\d\hd^2.
\end{align*}
$A$ is bounded linear, while $F$ and $G$ are proper convex lower semi-continuous.
Furthermore, $0\in\mathrm{int}\left(\mathrm{dom}\,G-A\,\mathrm{dom}\,F\right)$ (since $\mathrm{dom}\,G=\Y$)
so that by the Rockafellar--Fenchel duality theorem \cite[Thm.\,4.4.3]{BoZh2005} we have the strong duality
\begin{equation*}
\sup_{\phi\in\XX}-F(\phi)-G(A\phi)
=\inf_{w\in\Y'}F^*(A^*w)+G^*(-w),
\end{equation*}
and a minimizer $w$ of the right-hand side exists unless the above equals $-\infty$
($\Y'=\M(\overline\Omega\times[0,M])\times\M(\overline\Omega\times[0,M];\R^2\times\R)$ here denotes the dual space to $\Y$, and $F^*,G^*$ denote the convex conjugates of $F$ and $G$).
As calculated before, the left-hand side equals $\sup_{\phi\in\tilde\K} \inf_{v\in\C} \int_{\overline\Omega\times[0,M]} \phi\cdot\d Dv$,
so it remains to show $\inf_{w\in\Y'}F^*(A^*w)+G^*(-w)=\inf_{v\in\C} \sup_{\phi\in\tilde\K} \int_{\overline\Omega\times[0,M]} \phi\cdot\d Dv$.
We have
	\begin{align*}
	F^*(A^*w) 
	&= \sup_{\phi\in\XX} \langle\phi,A^*w\rangle - \iota_{\tilde\K}(\phi) = \sup_{\phi\in\XX} \langle A\phi,w\rangle - \iota_{\tilde\K}(\phi)
	= \sup_{\phi\in\tilde\K\cap\XX} \int_{\dom} \dive\phi \, \d w_1 + \int_{\dom}\phi\cdot\d w_2, \\
	G^*(w) &= \notinclude{G^*(w_1,w_2) = }\iota_{\mathcal{S}_1}(w_1) + \iota_{\mathcal{S}_2}(w_2)
	\end{align*} 
	with the sets
	\begin{align*}
	\mathcal{S}_1 &= \{ \mu\in\M(\dom) \ | \ \mu\ll\leb^3, \ 0\leq\mu\leq 1\}, \\
	\mathcal{S}_2 &= \{ -1_{u(\mu_+,\mu_-)}n \,\hd^2\restr\partial(\Omega\times(0,M))\}.
	\end{align*}
	Thus, $F^*(A^*w)+G^*(-w)\geq0$ for all $w\in\Y'$ so that the infimum over all $w$ is finite and $\inf_{w\in\Y'}F^*(A^*w)+G^*(-w)=\min_{w\in\Y'}F^*(A^*w)+G^*(-w)$.
	Furthermore we obtain
	\begin{multline*}
	\min_{w\in\Y'}F^*(A^*w)+G^*(-w)
	=\min_{-w\in\mathcal S_1\times\mathcal S_2}\sup_{\phi\in\tilde\K\cap\XX} \int_{\dom} \dive\phi \, \d w_1 + \int_{\dom}\phi\cdot\d w_2\\
	=\min_{w_1\in L^1(\Omega\times(0,M);[0,1])}\sup_{\phi\in\tilde\K\cap\XX}\int_{\partial(\Omega\times(0,M))}1_{u(\mu_+,\mu_-)}\phi\cdot n\,\d\hd^2- \int_{\Omega\times(0,M)} w_1\dive\phi \, \d x\,\d s.
	\end{multline*}
	Now the supremum on the right-hand side is only finite if $w_1$ is nonincreasing in $s$-direction.
	Indeed, finiteness of the supremum implies $\int_{\Omega\times(0,M)}w_1\partial_s\zeta\,\d x\,\d s\geq0$ for all $\zeta\in C_0^\infty(\Omega\times(0,M);[0,\infty))$
	since otherwise $\sup_{\phi\in\tilde\K\cap\XX}- \int_{\Omega\times(0,M)} w_1\dive\phi \, \d x\,\d s\geq\sup_{\lambda>0}-\int_{\Omega\times(0,M)} w_1\dive(0,0,\lambda\zeta)\,\d x\,\d s=\infty$.
	The fundamental lemma of the calculus of variations now implies that $w_1$ is nonincreasing in $s$-direction.
	Therefore, by approximating $w_1$ with its mollifications it is straightforward to see that $\int_{\Omega\times(0,M)}w_1\partial_s\zeta\,\d x\,\d s\leq\leb^2(\Omega)$
	for any $\zeta\in C_0^\infty(\Omega\times(0,M);[-1,1])$.
	As a consequence, we have
	\begin{align*}
	&\sup_{\phi\in\tilde\K\cap\XX}\int_{\partial(\Omega\times(0,M))}1_{u(\mu_+,\mu_-)}\phi\cdot n\,\d\hd^2- \int_{\Omega\times(0,M)} w_1\dive\phi \, \d x\,\d s\\
	&\geq\sup_{\phi\in\tilde\K\cap C_0^\infty(\Omega\times(0,M);\R^3)}- \int_{\Omega\times(0,M)} w_1\dive\phi \, \d x\,\d s\\
	&\geq\sup_{\phi\in\tilde\K\cap C_0^\infty(\Omega\times(0,M);\R^3),\zeta\in C_0^\infty(\Omega\times(0,M);[\frac{\tau(M)}M,\frac{\tau(M)}M])}- \int_{\Omega\times(0,M)} w_1\dive(\phi+(0,0,\zeta)) \, \d x\,\d s-\frac{\tau(M)}M\leb^2(\Omega)\\
	&\geq\sup_{\psi\in C_0^\infty(\Omega\times(0,M);\R^3),|\psi|\leq\frac{\tau(M)}M}- \int_{\Omega\times(0,M)} w_1\dive\psi\, \d x\,\d s-\frac{\tau(M)}M\leb^2(\Omega)\\
	&=\frac{\tau(M)}M|w_1|_{\mathrm{TV}}-\frac{\tau(M)}M\leb^2(\Omega),
	\end{align*}
	where $|\cdot|_{\mathrm{TV}}$ denotes the total variation seminorm.
	Thus the supremum is only finite if $w_1\in\BV(\Omega\times(0,1))$ so that we may write
	\begin{align*}
	&\min_{w\in\Y'}F^*(A^*w)+G^*(-w)\\
	&=\min_{w_1\in\BV(\Omega\times(0,M);[0,1])}\sup_{\phi\in\tilde\K\cap\XX}\int_{\partial(\Omega\times(0,M))}1_{u(\mu_+,\mu_-)}\phi\cdot n\,\d\hd^2- \int_{\Omega\times(0,M)} w_1\dive\phi \, \d x\,\d s\\
	&=\min_{w_1\in\BV(\Omega\times(0,M);[0,1])}\sup_{\phi\in\tilde\K\cap\XX}\int_{\partial(\Omega\times(0,M))}(1_{u(\mu_+,\mu_-)}-w_1)\phi\cdot n\,\d\hd^2+\int_{\Omega\times(0,M)}\phi\cdot\d Dw_1\\
	&=\min_{v\in\C}\sup_{\phi\in\tilde\K\cap\XX}\int_{\overline\Omega\times[0,M]}\phi\cdot\d Dv,
	\end{align*}
	where by density we may replace $\tilde\K\cap\XX$ with $\tilde\K$.
\end{proof} 

\begin{remark}[Predual variables of reduced regularity]
Since the predual objective functional $\D$ as well as the functional $\phi\mapsto\int_{\overline\Omega\times[0,M]}\phi\cdot\d Dv$ for $v\in\C$
are continuous with respect to the norm $\|\phi\|_{L^1}+\|\dive\phi\|_{L^1}$,
all throughout the statement of \cref{thm:strongDuality} the suprema may actually be taken over
\begin{multline*}
\hat\K = \Big\{ \phi=(\phi^x,\phi^s)\in L^1(\Omega\times(0,M);\R^2\times[0,\infty)) \,\Big|\, \dive\phi\in L^1(\Omega\times\R),\,\phi^s(x,s)\geq0,\\
|\phi^x(x,s)|\leq\tau'(0)\ \forall \ (x,s)\in\Omega\times(0,M),\
\textstyle\left|\int_{s_1}^{s_2} \phi^x(x,s)\d s\right| \leq \tau(s_2-s_1) \ \forall \ x\in\Omega, 0\leq s_1<s_2\leq M \Big\}.
\end{multline*}
\end{remark}

\subsection{Calibrations for simple network configurations}\label{sec:calibrations}

Even without knowing equality in \eqref{eqn:mainInequalities} one can make use of this inequality
and use it to prove optimality of given transport networks by providing a so-called calibration
(which is a predual certificate in the language of convex optimization).
In fact, this was the aim of introducing the functional lifting of $J$ in \cite{AlBoDa03}
(the authors even considered a Mumford--Shah inpainting setting as we have it here).
In this section we provide calibrations for two exemplary transport networks,
thereby showing optimality of these network configurations as well as equality in \eqref{eqn:mainInequalities} for these cases.
Throughout we will use the notation of the previous section.

\begin{lemma}[Predual optima]\label{thm:dualOptima}
For any $\phi\in\hat\K$ there exists a divergence-free $\hat\phi\in\hat\K$ with no smaller predual cost $\D$.
Thus, in the predual problem $\sup_{\phi\in\hat\K}\D(\phi)$ one may restrict to divergence-free vector fields $\phi$.
\end{lemma}
\begin{proof}
Let $\phi\in\hat\K$. 
By Smirnov's decomposition theorem \cite[Thm.\,B-C]{Sm93} there exists a set $S$ of simple oriented curves of finite length
(that is, measures of the form $\pushforward\gamma{\dot\gamma\hd^1\restr[0,1]}$ for $\gamma:[0,1]\to\overline\Omega\times[0,M]$ injective and Lipschitz,
where $\pushforward f\mu$ denotes the pushforward of a measure $\mu$ under a map $f$)
as well as a nonnegative measure $\rho$ on $S$ such that
\begin{equation*}
\phi=\int_S\tilde\phi\,\d\rho(\tilde\phi),\quad
\|\phi\|_{L^1}=\int_S\|\tilde\phi\|_\M\,\d\rho(\tilde\phi),\quad
\|\dive\phi\|_{\M}=\int_S\|\dive\tilde\phi\|_\M\,\d\rho(\tilde\phi)
\end{equation*}
(the first equation means $\langle\phi,\psi\rangle=\int_S\langle\phi,\psi\rangle\,\d\rho(\tilde\phi)$ for every smooth test vector field $\psi$
and $\langle\cdot,\cdot\rangle$ the duality pairing between Radon measures and continuous functions).
Now consider
\begin{equation*}
\tilde S=\{\pushforward\gamma{\dot\gamma\hd^1\restr[0,1]}\in S\,|\,\gamma(0)\in\Omega\times(0,M)\text{ or }\gamma(1)\in\Omega\times(0,M)\}
\end{equation*}
(which is $\rho$-measurable in the above sense),
then $\hat\phi=\int_{S\setminus\tilde S}\tilde\phi\,\d\rho(\tilde\phi)$ is divergence-free with $\D(\hat\phi)\leq\D(\phi)$ and $\hat\phi\in\hat\K$.
\end{proof}

The previous \namecref{thm:dualOptima} suggests to focus on divergence-free predual certificates, which in this context are called calibrations.

\begin{lemma}\label{lm:EqualityOfMinima}
If there exists a divergence-free predual certificate for $v\in\C$, that is, a vector field $\hat\phi\in\hat\K$ with
\begin{equation*}
\dive\hat\phi=0
\quad\text{and}\quad
\tilde\G(v)=\int_{\overline\Omega\times[0,M]}\hat\phi\cdot\d Dv,
\end{equation*}
then $v$ minimizes $\tilde\G$ on $\C$.
Moreover, $\hat\phi$ is a predual certificate for any minimizer.
In particular, if $v=1_u$ for some $u\in\A_u$ and thus $\E(\flux_u)=\tilde\E(u)=\G(1_u)=\tilde\G(1_u)=\int_{\overline\Omega\times[0,M]}\hat\phi\cdot\d D1_u$,
then $u$ minimizes $\E(\flux_u)=\tilde\E(u)=\G(1_u)$ over $\A_u$ and $\hat\phi$ is called a \emph{calibration} for $u$.
\end{lemma}
\begin{proof}
By weak duality from \cref{thm:strongDuality} we have $\tilde\G(v)\geq\D(\hat\phi)$ with equality if and only if $v$ and $\hat\phi$ are optimal.
However, since $\hat\phi$ is divergence-free we have
\begin{equation*}
\D(\hat\phi)
=\int_{\partial(\Omega\times(0,M))}1_{u(\mu_+,\mu_-)}\phi\cdot n\,\d\hd^2-\int_{\Omega\times(0,M)}v\dive\phi\,\d x\,\d s
=\int_{\overline\Omega\times[0,M]}\hat\phi\cdot\d Dv
=\tilde\G(v)
\end{equation*}
after an integration by parts, thus $v\in\C$ is minimizing and $\hat\phi\in\hat\K$ is maximizing.
Now any other minimizer $\tilde v\in\C$ satisfies $\tilde\G(\tilde v)=\tilde\G(v)=\D(\hat\phi)=\int_{\overline\Omega\times[0,M]}\hat\phi\cdot\d D\tilde v$ by the same calculation
so that $\hat\phi$ also calibrates $\tilde v$.
\end{proof}

\begin{remark}[Sequences as calibrations and less regularity]\label{rem:calibrationSequence}
By an obvious modification of the above argument, the existence of the divergence-free $\hat\phi\in\hat\K$ can be replaced
by the existence of a sequence $\phi_1,\phi_2,\ldots\in\hat\K$ of divergence-free vector fields with $\tilde\G(v)=\lim_{n\to\infty}\int_{\overline\Omega\times[0,M]}\phi_n\cdot\d Dv$.
\end{remark}

In the remainder of the section we provide two examples for calibrations,
one for a classic network configuration that can be and has been analysed classically on the level of graphs,
and one that cannot be analysed on such a basis.
We begin by proving the angle conditions for triple junctions, which, as mentioned above, can also easily be obtained by a vertex perturbation argument.
Any triple junction can locally be interpreted as having a single source point and two sink points (or vice versa), which we do below.

\begin{example}[Triple junction]\label{exm:tripleJunction}
Let a point source and two point sinks be located on the boundary of the unit disk $\Omega$,
\begin{equation*}
\mu_+=(m_1+m_2)\delta_{-e_0},\quad
\mu_-=m_1\delta_{e_1}+m_2\delta_{e_2}
\quad\text{for }e_0,e_1,e_2\in\partial\Omega=S^1,\,m_1,m_2>0,
\end{equation*}
where the vectors $e_0,e_1,e_2$ satisfy the angle condition
\begin{equation*}
0=\tau(m_1)e_1+\tau(m_2)e_2-\tau(m_1+m_2)e_0.
\end{equation*}
Then the mass flux $\flux=(m_1+m_2)e_0\hd^1\restr e_0+m_1e_1\hd^1\restr e_1+m_2e_2\hd^1\restr e_2$ minimizes $\E$ on $\A_\flux$.

To prove this statement, assume without loss of generality that $e_i=(\cos\varphi_i,\sin\varphi_i)$, $i=0,1,2$, with $0\leq\varphi_1\leq\varphi_0\leq\varphi_2\leq2\pi$ (see \cref{fig:tripleJunction}).
Then $u_\flux$ reads
\begin{equation*}
u_\flux(r,\varphi)=\begin{cases}
0&\text{if }\varphi\in[\varphi_2,2\pi-\varphi_0),\\
m_2&\text{if }\varphi\in[\varphi_1,\varphi_2),\\
m_1+m_2&\text{else}
\end{cases}
\end{equation*}
in polar coordinates, and its maximum is $M=m_1+m_2$. Now set
\begin{equation*}
\phi(x,s)=\begin{cases}
-\frac{\tau(m_2)}{m_2}(e_2^\perp,0)&\text{if }0\leq s\leq m_2,\\
-\frac{\tau(m_1)}{m_1}(e_1^\perp,0)&\text{if }m_2\leq s\leq M,
\end{cases}
\end{equation*}
where $\perp$ denotes counterclockwise rotation by $\pi/2$.
With this choice we have
\begin{equation*}
\int_{\overline\Omega\times[0,M]}\phi\cdot\d D1_{u_\flux}
=\tau(m_1)(1+e_0^\perp\cdot e_1^\perp)+\tau(m_2)(1+e_0^\perp\cdot e_2^\perp)
=\tau(m_1)+\tau(m_2)+\tau(m_1+m_2)
=\E(\flux),
\end{equation*}
where in the second equality we used the (inner product with $e_0$ of the) angle condition.
Furthermore, $\dive\phi=0$, and we have $\left|\int_{s_1}^{s_2}\phi^x(x,s)\,\d s\right|\leq\tau(s_2-s_1)$ for all $0\leq s_1\leq s_2\leq M$,
Indeed, for $s_2\leq m_2$ or $s_1\geq m_2$ this is trivial to check, and for $s_1\leq m_2\leq s_2$ we set $\alpha=\min\{\tfrac{m_2-s_1}{m_2},\tfrac{s_2-m_2}{m_1}\}$ and calculate
\begin{align*}
\left|\int_{s_1}^{s_2}\phi^x(x,s)\,\d s\right|
&=\left|(m_2-s_1)\tfrac{\tau(m_2)}{m_2}e_2+(s_2-m_2)\tfrac{\tau(m_1)}{m_1}e_1\right|\\
&=\left|\alpha(\tau(m_2)e_2+\tau(m_1)e_1)+(\tfrac{m_2-s_1}{m_2}-\alpha)\tau(m_2)e_2+(\tfrac{s_2-m_2}{m_1}-\alpha)\tau(m_1)e_1\right|\\
&\leq\alpha\tau(m_1+m_2)+(\tfrac{m_2-s_1}{m_2}-\alpha)\tau(m_2)+(\tfrac{s_2-m_2}{m_1}-\alpha)\tau(m_1)+(1+\alpha-\tfrac{m_2-s_1}{m_2}-\tfrac{s_2-m_2}{m_1})\tau(0)\\
&\leq\tau(\alpha(m_1+m_2)+(\tfrac{m_2-s_1}{m_2}-\alpha)m_2+(\tfrac{s_2-m_2}{m_1}-\alpha)m_1+(1+\alpha-\tfrac{m_2-s_1}{m_2}-\tfrac{s_2-m_2}{m_1})0)\\
&=\tau(s_2-s_1),
\end{align*}
where in the first inequality we used the triangle inequality and the angle condition
and in the last inequality we used Jensen's inequality
with convex combination coefficients $\alpha,(\tfrac{m_2-s_1}{m_2}-\alpha),\tfrac{s_2-m_2}{m_1}-\alpha),(1+\alpha-\tfrac{m_2-s_1}{m_2}-\tfrac{s_2-m_2}{m_1})$.
Thus, $\phi\in\hat\K$ as desired.
\end{example}

\begin{figure}
\centering
\setlength\unitlength{.07\linewidth}
\begin{picture}(2,2)
\put(1,1){\scalebox{1.5}{\circle{5}}}
\put(1,1){\vector(1,1){.7}}
\put(1,1){\vector(-1,1){.7}}
\put(1,0){\vector(0,1){1}}
\put(1.1,.5){\small$e_0$}
\put(1.4,1.2){\small$e_1$}
\put(.7,1.4){\small$e_2$}
\put(0,0){\small$\Omega$}
\put(.9,-.2){\small$+$}
\put(.1,1.75){\small$-$}
\put(1.7,1.75){\small$-$}
\end{picture}
\caption{Illustration of the notation in \cref{exm:tripleJunction}.}
\label{fig:tripleJunction}
\end{figure}
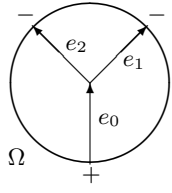

The second example shows that even for strictly concave transportation cost $\tau$ one may have a diffuse flux without network formation.

\begin{example}[Diffuse flux]
Let the source and the sink be two line measures opposite of each other, that is
\begin{equation*}
\mu_+=m\hd^1\restr[0,\ell]\times\{0\},\quad
\mu_-=m\hd^1\restr[0,\ell]\times\{d\},\quad
\text{for some }m,d,\ell>0\text{ and }\Omega=(0,\ell)\times(0,d).
\end{equation*}
By rescaling space and the transportation cost we may reduce the setting to the equivalent one with $m=d=1$ without loss of generality.
If the transportation cost $\tau$ satisfies
\begin{equation*}
\frac{\tau(m)}{\tau'(0)}\geq\max\left\{\min\{\tfrac m\beta,\tfrac12\},\sqrt{\beta^2-m^2}\arsinh\tfrac m{\sqrt{\beta^2-m^2}}\right\}
\quad\forall m<\beta
\end{equation*}
for some $\beta\geq1$ (note that necessarily $\tau(m)/\tau'(0)\leq m$), then the optimal flux is given by the diffuse $\flux=(0,1)\leb^2\restr\Omega$.

Note that for $\beta$ large enough, the above bound on $\tau$ simply evaluates to the strictly concave $\sqrt{\beta^2-m^2}\arsinh\tfrac m{\sqrt{\beta^2-m^2}}$;
from then on any larger $\beta$ produces a weaker bound.

To prove the statement note $u_\flux(x)=x_1$ with maximum $M=\ell$ and set
\begin{equation*}
\psi(x)=\begin{cases}
\frac{-x^\perp}{|x|}&\text{if }|x|\leq\frac12,\\
0&\text{else,}
\end{cases}
\qquad
\tilde\psi(x)=\left(\begin{smallmatrix}1&0\\0&1/\beta\end{smallmatrix}\right)\psi\left(\left(\begin{smallmatrix}1/\beta&0\\0&1\end{smallmatrix}\right)x\right),
\qquad
\phi(x,s)=\begin{cases}
(\tilde\psi(x_1-s,x_2),0)&\text{if }x_2\leq\frac12,\\
(\tilde\psi(s-x_1,1-x_2),0)&\text{else}
\end{cases}
\end{equation*}
(note that for each $s$, $\phi$ is symmetric about $x_2=\frac12$, describing an elliptic flow in each half).
It is straightforward to check
\begin{equation*}
\E(\flux)=\ell=\int_{\overline\Omega\times[0,M]}\phi\cdot\d D1_{u_\flux}
\end{equation*}
as well as $\dive\phi=0$.
Furthermore, we need to check the condition $\left|\int_{s_1}^{s_2}\phi^x(0,x_2,s)\,\d s\right|\leq\tau(s_2-s_1)$
for all $-\beta\sqrt{1/4-x_2^2}\leq s_1\leq s_2\leq\beta\sqrt{1/4-x_2^2}$ (outside this range $\phi$ is zero anyway), where due to symmetry it suffices to consider the position $x_1=0$.
We can calculate (without loss of generality for $x_2\leq\frac12$)
\begin{multline*}
\left|\int_{s_1}^{s_2}\phi^x(0,x_2,s)\,\d s\right|
=\left|\left(\begin{smallmatrix}1&0\\0&1/\beta\end{smallmatrix}\right)\int_{s_1}^{s_2}\psi(-s/\beta,x_2)\,\d s\right|
=\left|\int_{s_1}^{s_2}\frac{{x_2\choose s/\beta^2}}{\sqrt{x_2^2+s^2/\beta^2}}\,\d s\right|\\
=\left|x_2{\beta(\arsinh z_2)-\arsinh z_1))\choose\sqrt{1+z_2^2}-\sqrt{1+z_1^2}}\right|
=x_2\sqrt{\beta^2(\arsinh z_2-\arsinh z_1)^2+\left(\sqrt{1+z_2^2}-\sqrt{1+z_1^2}\right)^2}
\end{multline*}
for $z_i=s_i/(x_2\beta)$, $i=1,2$. Let us abbreviate this function by $f(s_1,s_2,x_2,\beta)$.
We need to have $\tau(m)\geq f(s-\frac m2,s+\frac m2,x_2,\beta)$ for any choice of $s$ (which due to symmetry we may assume nonnegative) and $x_2$.
Now it turns out that $f(s-\frac m2,s+\frac m2,x_2,\beta)$ has no critical points as a function of $s$ and $x_2$. Indeed,
\begin{align*}
&\frac\d{\d x_2}f(s-\tfrac m2,s+\tfrac m2,x_2,\beta)\\
&=-x_2\beta^2\frac{(\arsinh z_2-\arsinh z_1)^2-(\arsinh z_2-\arsinh z_1)\left(\frac{z_2}{\sqrt{1+z_2^2}}-\frac{z_1}{\sqrt{1+z_1^2}}\right)+\left(2-\sqrt{\frac{1+z_2^2}{1+z_1^2}}-\sqrt{\frac{1+z_1^2}{1+z_2^2}}\right)/\beta^2}{f(s-\tfrac m2,s+\tfrac m2,x_2,\beta)},\\
&\frac\d{\d s}f(s-\tfrac m2,s+\tfrac m2,x_2,\beta)\\
&=\beta\frac{(\arsinh z_2-\arsinh z_1)\left(\frac1{\sqrt{1+z_2^2}}-\frac1{\sqrt{1+z_1^2}}\right)+(\sqrt{1+z_2^2}-\sqrt{1+z_1^2})\left(\frac{z_2}{\sqrt{1+z_2^2}}-\frac{z_1}{\sqrt{1+z_1^2}}\right)/\beta^2}{f(s-\tfrac m2,s+\tfrac m2,x_2,\beta)},
\end{align*}
and one can check that there are no joint zeros $(z_1,z_2)$ of both expressions.
Consequently, $f(s-\tfrac m2,s+\tfrac m2,x_2,\beta)$ becomes extremal on the boundary of the admissible domain $\{(s,x_2)\in\R\times[0,\infty)\,|\,\frac{(s\pm m/2)^2}{\beta^2}+x_2^2\leq\frac14\}$
(such that $s_2=s+\tfrac m2\leq\beta\sqrt{1/4-x_2^2}$).
One can readily evaluate $f(s-\tfrac m2,s+\tfrac m2,0,\beta)=\frac{||s_2|-|s_1||}\beta\leq\min\{\frac m\beta,\frac12\}$, thus we require $\tau(m)\geq\min\{\frac m\beta,\frac12\}$.
On the other boundary, $x_2=\sqrt{\frac14-\frac{(s\pm m/2)^2}{\beta^2}}$, for symmetry reasons it suffices if we consider $s\geq0$.
It turns out that $f(s-\tfrac m2,s+\tfrac m2,\sqrt{\frac14-\frac{(s+m/2)^2}{\beta^2}},\beta)$ is initially decreasing in $s\in[0,\frac{\beta-m}2]$
and may then again increase, depending on the size of $\beta$ and $m$.
Thus the maximum value is taken either at $s=\frac{\beta-m}2$ (which is the case $x_2=0$ already treated above) or at $s=0$.
Hence, we additionally need $\tau(m)\geq f(-\tfrac m2,\tfrac m2,\sqrt{\frac14-\frac{m^2}{4\beta^2}},\beta)=\sqrt{\beta^2-m^2}\arsinh\tfrac m{\sqrt{\beta^2-m^2}}$
so that $\phi\in\hat\K$ as desired.
\end{example}

\section{Adaptive finite elements for functional lifting problems}\label{sec:numerics}

Convex optimization problems arising from functional lifting as introduced in \cref{sec:FuncLifting} require a careful numerical treatment due to several reasons.
First, the lifted problem has an objective variable living in three rather than two space dimensions,
which requires a careful discretization in order to provide a straightforward translation between the two- and three-dimensional model.
Furthermore, the problem size is strongly increased by the lifting;
not only do the variables live in a higher-dimensional space, but also the set $\hat\K$ has a constraint for every $(x,s_1,s_2)\in\Omega\subset\R^2\times\R\times\R$
so that the problem essentially behaves like a four-dimensional one.
Finally, to make the algorithm reliable and avoid unwanted effects the discretization of the feasible set $\hat\K$ should be feasible itself (that is, a subset of $\hat\K$)
which means that one must be able to reduce the infinite number of non-local constraints in $\hat\K$ to a finite number.

One possible way to jointly tackle the previously mentioned challenges is an adaptive finite element approach defined on grids consisting of prism-shaped elements.
As before, to emphasize the difference between the original image domain and the image range, for a point in $\Omega\times(0,M)$, we denote its first two coordinates as $x$-coordinates and the third one as $s$-coordinate with respect to the standard basis of $\R^3$.

\subsection{Adaptive triangular prism grids}

We start by recalling the definition of a two-dimensional simplicial grid (see for instance \cite{Tr1997}). 

\begin{definition}[Simplex, simplicial grid in 2D]
	A \emph{two-dimensional simplex} $(x^0,x^1,x^2)$ is a 3-tuple with nodes  $x^0,x^1,x^2\in\R^2$, which do not lie on a one-dimensional hyperplane. The convex hull $\text{conv}\{x^0,x^1,x^2\}$ is also denoted as a simplex. A \emph{two-dimensional simplicial grid} on $\overline\Omega$ is a set of two-dimensional simplices with pairwise disjoint interior and union $\overline\Omega$.
\end{definition}

Based on a two-dimensional simplicial grid for the image domain $\Omega$, we define a \emph{lifted counterpart} consisting of triangular prism-shaped elements.
For two tuples $(x^0,\ldots,x^k)$ and $(s^0,\ldots,s^l)$ of points in $\R^n$ and $\R^m$ we will write $(x^0,\ldots,x^k)\times(s^0,\ldots,s^l)$
for the tuple $((x^0,s^0),(x^1,s^0),\ldots(x^k,s^l))$ of points in $\R^n\times\R^m$.

\begin{definition}[Triangular prism element]
A \emph{triangular prism element} $T$ is a 6-tuple $T_x\times T_s$ of nodes in $\R^3$, where $T_x=(x^0,x^1,x^2)$ is a two-dimensional simplex and $T_s=(s^0,s^1)$ for $s^0,s^1\in\R$ with $s^1>s^0$.
If there is no ambiguity, the convex hull of the nodes is also denoted a triangular prism element
(and $T_x$ and $T_s$ are likewise identified with their corresponding convex hulls).
The \emph{vertical} and \emph{horizontal edges} of $T$ are given by $x^i\times(s^0,s^1)$ and $(x^i,x^j)\times s^k$, respectively, for $i,j\in\{0,1,2\}$, $k\in\{0,1\}$.
Similarly one defines its \emph{vertical} and \emph{horizontal faces}.
\end{definition} 

A single triangular prism element can be refined either in the $x$-plane or in the $s$-direction as illustrated in \cref{fig:LocalRefinement},
where we suggest to use the obvious extension of the standard bisection method for a two-dimensional simplicial grid (see for instance \cite{Tr1997}).

\begin{definition}[Element refinement]
The \emph{$s$-refinement} of a triangular prism element $T=(x^0,x^1,x^2)\times(s^0,s^1)$ is the pair of triangular prism elements
\begin{equation*}
(x^0,x^1,x^2)\times(s^0,\tfrac{s^0+s^1}2)\text{ and }(x^0,x^1,x^2)\times(\tfrac{s^0+s^1}2,s^1).
\end{equation*}
Assuming without loss of generality $(x^1,x^2)$ to be the longest edge of $(x^0,x^1,x^2)$,
the \emph{$x$-refinement} of $T$ is the pair of triangular prism elements
\begin{equation*}
(x^0,x^1,\tfrac{x^1+x^2}{2})\times(s^0,s^1)\text{ and }(x^0,\tfrac{x^1+x^2}{2},x^2)\times(s^0,s^1).
\end{equation*}
\end{definition}

\begin{figure}
	\centering
	\setlength\unitlength\textwidth
	\begin{picture}(.5,.2)
	\put(0,0){\includegraphics[height=0.2\unitlength,trim=58 0 221 14,clip]{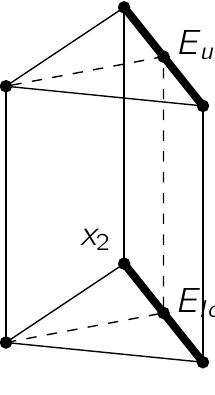}}
	\put(.4,0){\includegraphics[height=0.2\unitlength,trim=218 0 60 14,clip]{LocalRefinement.pdf}}
	\put(.04,.07){\color{white}\rule{1.7ex}{1.25ex}}
	\put(.09,.168){\color{white}\rule{2ex}{1.8ex}}
	\put(.09,.04){\color{white}\rule{1.3ex}{1.6ex}}
	\put(.105,.035){\color{white}\rule{1.3ex}{1.6ex}}
	\put(.44,.07){\color{white}\rule{1.7ex}{1.25ex}}
	\end{picture}
	\caption{Subdivision of a triangular prism element $T$ by $x$-refinement along the longest (bold) horizontal edge (left) and $s$-refinement(right).}
	\label{fig:LocalRefinement}
\end{figure}

We aim for simulations on an adaptively refined grid.
During refinement we want to keep a certain regularity condition of the grid which we call semi-regular.

\begin{definition}[Triangular prism grid and hanging nodes]
A \emph{triangular prism grid} $\T$ on $\overline\Omega\times[0,M]$ is a set of triangular prism elements with pairwise disjoint interior and union $\overline\Omega\times[0,M]$.
Its \emph{set of nodes} $\N(\T)$ is the union of the nodes of all its elements.

A node $N\in\N(\T)$ is called \emph{hanging} if there is an element $T\in\T$ with $N\in T$, but $N$ is not a node of $T$.
It is \emph{$s$-hanging} (or \emph{$x$-hanging}) if for any such $T$ the node $N$ lies on a vertical (or horizontal) edge of $T$.

The grid $\T$ is called \emph{regular} if it does not contain any hanging nodes.
It is called \emph{semi-regular} if it does not contain any $x$-hanging nodes
and if any two elements $T=(x^0,x^1,x^2)\times(s^0,s^1),S=(y^0,y^1,y^2)\times(r^0,r^1)\in\T$ with nonempty intersection 
either exactly share a node, an edge or a face or satisfy either $s^0,s^1\in\{r^0,\frac{r^0+r^1}2,r^1\}$ or $r^0,r^1\in\{s^0,\frac{s^0+s^1}2,s^1\}$.
\end{definition} 

Obviously, in addition to sharing a full edge or face, neighbouring elements in a semi-regular prism grid may also be such
that a vertical edge or face of one may be a vertical half-edge or half-face of the other, as illustrated in \cref{fig:SemiRegularTriangularPrismGrid},
resulting in $s$-hanging nodes.
The limitation of $s$-hanging nodes to one per edge is a natural convention to prevent too many successive hanging nodes, which are typically not associated with any degrees of freedom.
The $x$-refinement only allows bisection of the longest edge which is the standard means to prevent degeneration of the interior element angles.

\begin{figure}
	\centering
	\includegraphics[width=0.8\textwidth]{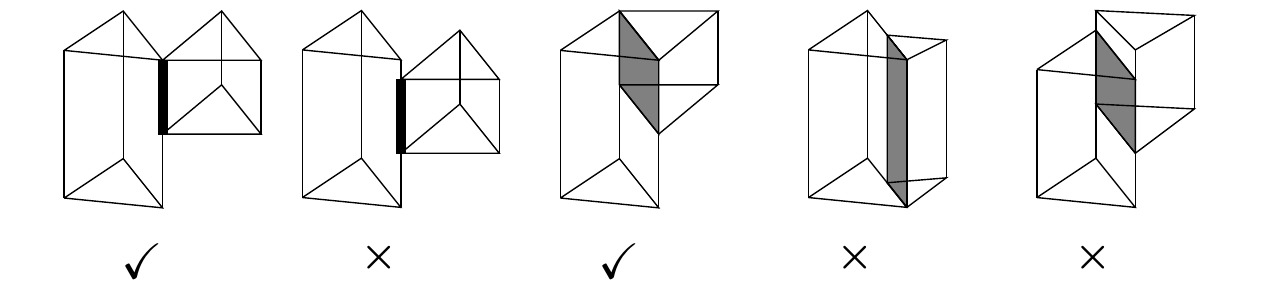}
	\caption{Examples of allowed and forbidden neighbouring relations in a semi-regular triangular prism grid.}
	\label{fig:SemiRegularTriangularPrismGrid}
\end{figure}

The rationale behind concentrating on semi-regular grids is
that these allow a simple discretization of the set of lifting constraints (as will be detailed in \cref{sec:constraintReduction})
and at the same time are sufficiently compatible with local refinement.
Indeed, had we only admitted regular grids, then any $s$-refinement would have to be done globally for all elements in a two-dimensional cross-section of the grid,
while the possibility of $s$-hanging nodes in semi-regular grids allows to subdivide just a few local elements in $s$-direction.
On the other hand, $x$-refinement can be done locally at a position $\hat x$ in the $x$-plane,
but has to be performed simultaneously for all elements along the $s$-coordinate sitting above $\hat x$.
However, this is just global refinement along a one-dimensional direction (rather than the above-mentioned global refinement in a two-dimensional cross-section),
and due to the possibility of local $s$-refinement one practically only has quite few elements along this direction.

A suitable algorithm for grid refinement should preserve the semi-regularity of the grid. Thus, the refinement of one element potentially implies the successive refinement of several neighbouring elements. In case of $x$-refinement, this affects all elements sharing a bisected face or edge with the refined element (that is, the element above and below as well as the neighbour across the subdivided vertical face). In case of $s$-refinement, the \emph{half-edge} rule has to be maintained, such that horizontal neighbours whose height exceeds twice the height of the refined element need to be refined successively. It is a standard fact that the resulting chains of successive element refinements terminate after a finite number of steps.


%

Finally, we note that the projection of a semi-regular triangular prism grid onto the $x$-hyperplane $\R^2\times\{0\}$ naturally yields a two-dimensional simplicial grid by construction, and so does every horizontal slice of the grid.

\subsection{Reduction of the constraint set $\hat\K$}\label{sec:constraintReduction}

Having fixed the grid, we now need to discretize functions on that grid.
We will choose these functions to be piecewise linear in $x$-direction and piecewise constant in $s$-direction (the details are given in \cref{sec:FE}).
In this section we give the reason for that choice: It easily allows to check and project onto the conditions in the convex set $\hat\K$.
A priori, this is very challenging, since for every base point $x\in\Omega$ we have an infinite number of inequality constraints.
Furthermore, after discretization, the inequality constraints for different base points might interdepend on each other in a nontrivial way due to interpolation between different nodal values.
We first show that for functions piecewise constant along the lifting dimension the infinite number of inequality constraints at each base point $x\in\Omega$ reduces to a finite number.
We then prove that if the functions are piecewise linear in $x$-direction, only the constraints for nodal base points have to be checked.

\begin{theorem}[Constraint set for functions piecewise constant in $s$] \label{thm:ConstraintsForConstantFunctions}
Let $0=t_0<t_1<\ldots<t_p=M$ be a partition of $[0,M]$ and let $\psi:[0,M)\to\R^2$ be piecewise constant,
\begin{equation*}
\psi(s)=C_i\text{ if }s\in[t_i,t_{i+1}),\,i=0\ldots p-1.
\end{equation*}
Let $\tau$ be a transportation cost. We have
\begin{equation*}
\left|\int_{s_1}^{s_2}\psi\,\d s\right|\leq\tau(|s_2-s_1|)\,\forall s_1,s_2\in[0,M]
\quad\text{if and only if}\quad
\left|\int_{s_1}^{s_2}\psi\,\d s\right|\leq\tau(|s_2-s_1|)\,\forall s_1,s_2\in\{t_0,\ldots,t_p\}.
\end{equation*}
\end{theorem}

\begin{proof}
We only need to prove one implication (the other being trivial).
Let $|\int_{s_1}^{s_2}\psi\,\d s|\leq\tau(|s_2-s_1|)$ for all $s_1,s_2\in\{t_0,\ldots,t_p\}$.
Now fix arbitrary $s_1,s_2\in[0,M]$, where without loss of generality we have $s_1<s_2$.
If $s_1,s_2\in[t_i,t_{i+1}]$ for some $i\in\{0,\ldots,p\}$, then
\begin{equation*}
\left|\int_{s_1}^{s_2}\psi\,\d s\right|
=(s_2-s_1)|C_i|
=\frac{s_2-s_1}{t_{i+1}-t_i}\left|\int_{t_i}^{t_{i+1}}\psi\,\d s\right|
\leq\frac{s_2-s_1}{t_{i+1}-t_i}\tau(t_{i+1}-t_i)
\leq\tau(s_2-s_1)
\end{equation*}
due to $\tau(0)=0$ and the concavity of $\tau$.
It remains to consider the case $s_1\in[t_i,t_{i+1}]$ and $s_1\in[t_j,t_{j+1}]$ with $i<j$.
To this end consider the function $f:[t_i,t_{i+1}]\times[t_j,t_{j+1}]\to\R$,
\begin{equation*}
f(s_1,s_2)
=\left|\int_{s_1}^{s_2}\psi\,\d s\right|-\tau(s_2-s_1)
=\left|(t_{i+1}-s_1)C_i+\int_{t_{i+1}}^{t_j}\psi\,\d s+(s_2-t_j)C_j\right|-\tau(s_2-s_1).
\end{equation*}
As a composition of a convex with an affine function, $f$ is jointly convex in both arguments.
Therefore, since $f\leq0$ at the four corners (the convex extreme points) of its domain, we have $f\leq0$ all over the domain, which finishes the proof.
\end{proof}

As a consequence, a piecewise constant approximation of the variables in the lifted direction allows an efficient constraint handling.
This feature breaks down already for piecewise linear instead of piecewise constant functions (where it becomes much harder to check the constraints),
as the following simple counterexample illustrates.

\begin{example}[Constraint set for functions piecewise linear in $s$]\label{exm:LinearBasisDoesNotWork}
Let $p\in\mathbb{N}$, $h_s=\frac{M}{p}$, and $t_i=ih_s$ for $i=0,\ldots,p$. Fix an arbitrary $C\in\R^2$ and define $\psi:[0,M]\to\R^2$ as 
\begin{equation*}
\psi(x) = \begin{cases} \frac{2C}{h_s}(s-t_i) - C & \text{ if } s\in [t_i,t_{i+1}],\,i \text{ even, } \\
\frac{2C}{h_s}(t_i-s) + C & \text{ if } s\in [t_i,t_{i+1}],\,i \text{ odd } \end{cases}
\end{equation*}
(see \cref{fig:ProfilesOfPhi}). Then obviously $|\int_{t_i}^{t_{j}} \psi\, \d s| = 0$ for any $i,j\in\{0,\ldots,p\}$,
while for $\tilde{s} = \frac{t_i+t_{i+1}}{2}$ we have
\begin{equation*}\textstyle
\left|\int_{t_i}^{\tilde{s}} \psi\, \d s \right| = \tfrac{h_s}{4}|C|
\end{equation*}
which can be arbitrarily large depending on $C$.
\end{example}

\begin{figure}
    \centering
	\setlength\unitlength\textwidth
	\includegraphics[width=0.5\unitlength]{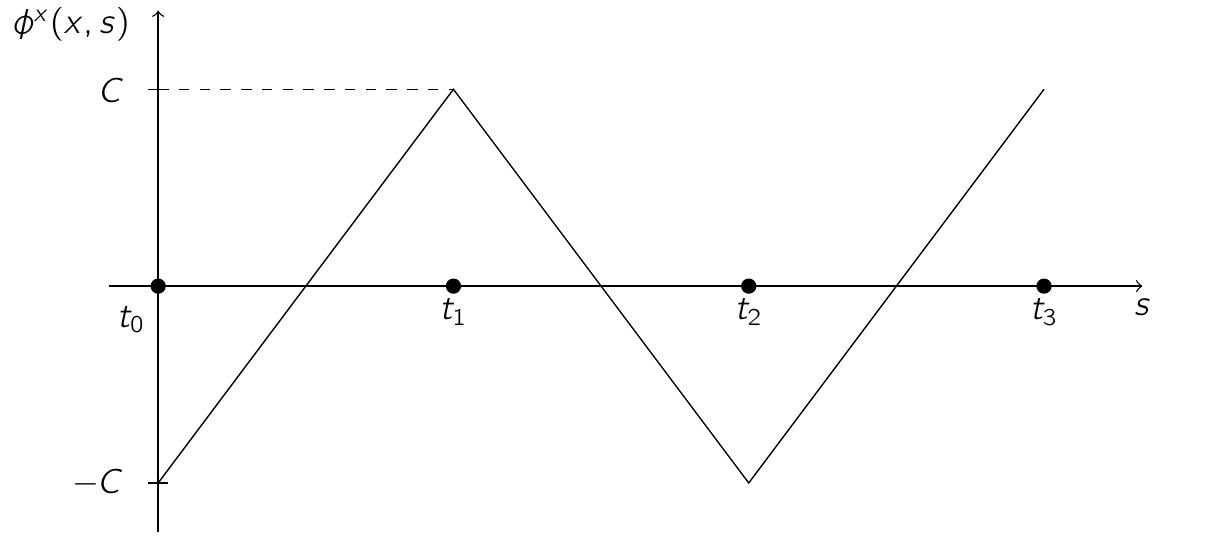}
	\begin{picture}(0,0)
	\put(-.505,.2){\color{white}\rule{5ex}{2.5ex}}
	\put(-.465,.215){\small$\psi$}
	\end{picture}
	\caption{
	Sketch of the function $\psi$ from \cref{exm:LinearBasisDoesNotWork}.}
	\label{fig:ProfilesOfPhi}
\end{figure}
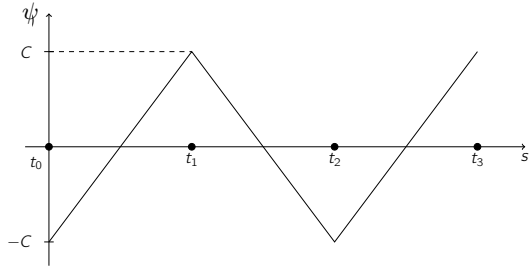

We next state that for piecewise linear discretization in $x$-direction it suffices to consider a finite number of base points.

\begin{theorem}[Constraint set for functions piecewise linear in $x$]\label{thm:pwLinearAnsatz}
Let $\tilde\T$ be a regular two-dimensional simplex grid on $\overline\Omega$ with node set $\N(\tilde\T)$,
and let $\phi:\overline\Omega\times[0,M]\to\R^3$ be piecewise linear in $x$-direction, that is,
for each $s\in[0,M]$ the function $x\mapsto\phi(x,s)$ is continuous and affine on each simplex $T\in\tilde\T$. Then
\begin{equation*}
\left|\int_{s_1}^{s_2}\phi^x(x,s)\,\d s\right|\leq\tau(|s_2-s_1|),\
|\phi^x(x,s_1)|\leq\tau'(0),
\text{ and }\phi^s(x,s_1)\geq0\,\forall s_1,s_2\in[0,M]
\end{equation*}
is satisfied for all $x\in\overline\Omega$ if and only if it is satisfied for all $x\in\N(\tilde\T)$.
\end{theorem}
\begin{proof}
Again, one implication is trivial, and we show the other one.
Let the constraints be satisfied for all $x\in\N(\tilde\T)$.
Now pick an arbitrary $x\in\overline\Omega$ and let $T=(x^0,x^1,x^2)\in\tilde\T$ such that $x\in T$
and thus $x=\lambda_0x_0+\lambda_1x_1+\lambda_2x_2$ for convex combination coefficients $\lambda_0,\lambda_1,\lambda_2\in[0,1]$.
Now the function $\phi(x,\cdot)$ can be written as the convex combination $\phi(x,\cdot)=\lambda_0\phi(x_0,\cdot)+\lambda_1\phi(x_1,\cdot)+\lambda_2\phi(x_2,\cdot)$.
Since the constraints are convex in $\phi(x,\cdot)$ and are satisfied for $\phi(x_0,\cdot)$, $\phi(x_1,\cdot)$, and $\phi(x_2,\cdot)$,
they are also satisfied for $\phi(x,\cdot)$.
\end{proof}

Note that the important feature of the piecewise linear discretization in $x$-direction which allows the above constraint reduction
is that the nodal basis of each element is a nonnegative partition of unity and can thus at each point $x\in\overline\Omega$ be viewed as a set of convex combination coefficients.
This feature breaks down for higher order elements.

Summarizing, if the ($x$-component of the) flux $\phi$ is discretized as piecewise constant in $s$-direction and piecewise linear in $x$-direction,
then the constraints forming the set $\hat\K$ only need to be checked at all nodes of the underlying grid.

The above also explains why we aim for semi-regular grids and avoid $x$-hanging nodes:
Otherwise, one would have to test the constraints also for all base points $\hat x$ that correspond to $x$-hanging nodes
(and over these points one would need to consider all $s_1,s_2\in[0,M]$ at which there is an element face, not only those $s_1,s_2\in[0,M]$ for which $(\hat x,s_1)$ and $(\hat x,s_2)$ are nodes).
Furthermore, the projection of a discretized vector field $\phi$ onto the constraint set will be much more complicated:
Without $x$-hanging nodes one can perform the projection independently for all nodes of the underlying two-dimensional simplex grid.
With $x$-hanging nodes, however, the constraints are no longer independent, since the function value at a hanging node is slaved to the function values at the neighbouring non-hanging nodes.


\subsection{Finite element discretization}\label{sec:FE}

We now aim to discretize our convex saddle point problem
\begin{equation}\label{eq:ConvexProblem}
\inf_{v\in\C}\sup_{\phi\in\hat\K}\int_{\overline\Omega\times[0,M]}\phi\cdot\d Dv
\end{equation}
based on a triangular prism finite element approach.
Motivated by \cref{thm:ConstraintsForConstantFunctions,thm:pwLinearAnsatz}, on a semi-regular triangular prism grid $\T$ we define the discrete function spaces
\begin{align*}
S^1(\T)&=\{w\in C^0(\overline\Omega\times[0,M]\,|\,w|_T\text{ is affine}\},\\
S^{0,1}(\T) &= \{ w:\overline\Omega\times[0,M)\rightarrow\R \, |\, w(\cdot,s)\in C^0(\overline\Omega)\text{ for all }s\in[0,M)\\
&\hspace*{7ex}\text{and for all }T\in\T\text{ there are }a,b,c\in\R\text{ with }w|_{T'}(x,s)=ax_1+bx_2+c\},
\end{align*}
where $w|_{T}$ denotes the restriction of $w$ onto $T$ and where $T'$ denotes the triangular prism element without its upper triangular face.
Obviously, on $T=(x^0,x^1,x^2)\times(s^0,s^1)$ any function $w\in S^1(\T)$ is uniquely determined by its values at the element nodes,
while $w|_{T'}$ for $w\in S^{0,1}(\T)$ is uniquely determined by the values of $w$ at the bottom nodes $(x^0,s^0),(x^1,s^0),(x^2,s^0)$.
Consequently, any $w\in S^{1}(\T)$ is uniquely determined by its function values at the set of all except the hanging nodes, which we denote by $\dofnodes(\T)\subset\N(\T)$,
and any $w\in S^{0,1}(\T)$ is uniquely determined by its function values at the set of all except the hanging and the top-most nodes, which we denote by $\dofnodes'(\T)\subset\dofnodes(\T)$.
Numbering the nodes in $\dofnodes(\T)$ and $\dofnodes'(\T)$ as $N_1,\ldots,N_{q'}$ and $N_1,\ldots,N_{q''}$, respectively,
we can thus define a nodal basis $(\theta_1,\ldots,\theta_{q'})$ of $S^{1}(\T)$ and $(\psi_1,\ldots,\psi_{q''})$ of $S^{0,1}(\T)$ via
\begin{equation*}\label{eq:Basis}
\theta_i(P) = \begin{cases} 1 & \text{ if } P=N_i, \\ 0 & \text{ otherwise } \end{cases}
\text{ for all }P\in\dofnodes(\T),\quad
\psi_i(P) = \begin{cases} 1 & \text{ if } P=N_i, \\ 0 & \text{ otherwise } \end{cases} 
\text{ for all }P\in\dofnodes'(\T).
\end{equation*}
We aim for a conformal discretization, that is, our discretized primal and dual variables $v^h,\phi^h$ shall satisfy $v^h\in\C$ and $\phi^h\in\hat\K$.
Therefore we choose $v^h,(\phi^h)^x\in S^{0,1}(\T)$ and $(\phi^h)^s\in S^1(\T)$ so that $v^h$ and $\phi^h$ can be written in terms of basis functions as 
\begin{equation*}
v^h(x,s) = \sum_{k=1}^{q''}V_k\psi_k(x,s),
\quad
\phi^h(x,s) = \left(\sum_{k=1}^{q''}\Phi_k^1\psi_k(x,s),\sum_{k=1}^{q''}\Phi_k^2\psi_k(x,s),\sum_{k=1}^{q'}\Phi_k^s\theta_k(x,s)\right),
\end{equation*}
where we denoted the corresponding vectors of nodal function values by capital letters $V,\Phi^1,\Phi^2\in\R^{q''}$, $\Phi^s\in\R^{q'}$.

\begin{remark}[Handling of top domain boundary]
In the continuous saddle point problem \eqref{eq:ConvexProblem}, the cost functional also includes the integral of the primal and dual function on the top domain boundary $\overline\Omega\times\{M\}$,
however, we chose to define our discrete functions in $S^{0,1}(\T)$ only on $\overline\Omega\times[0,M)$.
This is unproblematic since in \eqref{eq:ConvexProblem} we may replace $\overline\Omega\times[0,M]$ with $\overline\Omega\times[0,M)$ without changing the problem:
Since $v=0$ on $\overline\Omega\times[0,M)$ and thus necessarily $D_xv\restr\overline\Omega\times\{M\}=0$ and $D_sv\restr\overline\Omega\times\{M\}\leq0$, we have
\begin{equation*}
\int_{\overline\Omega\times\{M\}}\phi\cdot\d Dv=\int_{\overline\Omega\times\{M\}}\phi^s\d D_sv\leq0.
\end{equation*}
If this were strictly smaller than zero, then by decreasing $\phi^s$ to zero in a small enough neighbourhood of $\overline\Omega\times\{M\}$
we could increase $\int_{\overline\Omega\times[0,M]}\phi\cdot\d Dv$
so that in the supremum in \eqref{eq:ConvexProblem} we may indeed ignore the contribution from $\overline\Omega\times\{M\}$ without changing its value.

Another way to view this is the observation that $\phi^s(\cdot,M)$ is nothing else but the Lagrange multiplier for the constraint that $v$ must be decreasing in $s$-direction at $s=M$,
which however is automatically fulfilled due to the conditions $v\geq0$ and $v(\cdot,s)=0$.

Note that an alternative would have been to introduce an auxiliary layer of triangular prism elements right above $\overline\Omega\times[0,M)$
so that the discretized functions also have a well-defined value on $\overline\Omega\times\{M\}$.
\end{remark}

\begin{remark}[Approximability of the functional]
If the triangular prism grid is refined one can approximate a continuous function $v$ by discrete functions $v^h$ in the weak-* sense.
Note that for a reasonable approximation of functionals involving $Dv$ (as in our case) this is usually not sufficient;
instead one typically needs $v^h$ to approximate $v$ in the sense of strict convergence (in which additionally $\|Dv^h\|_\M\to\|Dv\|_\M$).
Unfortunately, this is not possible with a piecewise constant discretization,
however, for the special structure of our functional this would be asking a little bit too much.
Indeed, considering for simplicity $v=1_u$, the cost function satisfies
$\G(1_u)=\tilde\E(u)=\lim_{n\to\infty}\tilde\E(u_n)=\lim_{n\to\infty}\G(1_{u_n})$
for some sequence $u_n$ of piecewise constant images (which follows from \cref{def:ContinuousCostFunctionals} of $\E(\flux_u)=\tilde\E(u)$ as the relaxation of the cost for discrete mass fluxes).
Thus, $\G$ can be well approximated even with a discretization that is piecewise constant on a triangular prism grid $\T$ in $s$-direction.
The $x$-derivative of $v$ has to be better resolved, though, in order to be able to correctly account for the lengths of all network branches.
This means we require strict convergence in $x$-direction, $\|D_xv^h\|_\M\to\|D_xv\|_\M$, and this is indeed ensured by our piecewise linear discretization in $x$-direction.
The discretization of the fluxes $\phi$ now is dual to the one of $v$
in the sense that the divergence of $\phi^h$ is also piecewise constant in $s$-direction and piecewise linear in $x$-direction.
Thus it turns out that from the point of view of the underlying functional lifting, the proposed discretization is a quite natural, conformal one.
\end{remark}

Based on this finite element discretization, we can now reformulate the convex saddle point problem \eqref{eq:ConvexProblem} in terms of the coefficient vectors $V,\Phi^1,\Phi^2,\Phi^s$ as
\begin{equation*}\label{eq:DiscreteConvexProblem}
\min_{V\in\C^h(\T)}\max_{(\Phi^1,\Phi^2,\Phi^s)\in\hat\K^h(\T)}V\cdot M^1\Phi^1+V\cdot M^2\Phi^2+V\cdot M^s\Phi^s,
\end{equation*}
where $\C^h(\T)$ and $\hat\K^h(\T)$ are the sets of coefficient vectors corresponding to all functions in $\C\cap S^{0,1}(\T)$ and $\hat\K\cap(S^{0,1}(\T)\times S^{0,1}(\T)\times S^{1}(\T))$, respectively,
and where $M^1,M^2,M^s$ denote the mixed mass-stiffness matrices
\begin{align*}
M^1_{kl} &= \int_{\overline\Omega\times[0,M)}\psi_l\tfrac{\partial\psi_k}{\partial x_1} \, \d x \,\d s,  \ M^2_{kl} = \int_{\overline\Omega\times[0,M)}\psi_l\tfrac{\partial\psi_k}{\partial x_2} \, \d x \,\d s,  \
M^s_{kl} = \int_{\overline\Omega\times[0,M)}\psi_l\tfrac{\partial\theta_k}{\partial s}\, \d x \,\d s.
\end{align*}
In order to explicitly express $\C^h(\T)$ and $\hat\K^h(\T)$ we abbreviate
\begin{equation*}
L_{x,s^1,s^2}
=\left\{(x,s)\in\dofnodes(\T)\,\middle|\,s^1\leq s<s^2\right\}
\end{equation*}
to be the non-hanging nodes with $x$-coordinate $x$ and $s$-coordinate between $s^1$ and $s^2$.
Then we can write
\begin{align*}
\C^h(\T) &= \left\{V\in[0,1]^{q''}\,\middle|\,V_k=1_{u(\mu_+,\mu_-)}(N_k)\text{ for all }k\in\{1,\ldots,q''\}\text{ with }N_k\in\dofnodes'(\T)\cap\partial(\Omega\times(0,M))\right\}, \\
\hat\K^h(\T) &= \left\{\vphantom{\textstyle\sum_{N_k\in L_{x,s^1,s^2}}}(\Phi^1,\Phi^2,\Phi^s)\in(\R^{q''})^2\times\R^{q'}\,\middle| \, \Phi^s_k\geq 0\notinclude{, \ |(\Phi^1_l,\Phi^2_l)|\leq\tau'(0)}\ \forall k=1,\ldots,q',\notinclude{\ l=1,\ldots,q'', }\right.\\
&\textstyle\hspace*{25ex}\left.\left| \sum_{N_k\in L_{x,s^1,s^2}} h(N_k)(\Phi^1_k,\Phi^2_k) \right| \leq \tau(|s^2-s^1|) \ \forall \ (x,s^1),(x,s^2)\in \dofnodes(\T) \right\},
\end{align*}
where $h(N)$ is the distance of $N\in L_{x,s^1,x^2}$ to the next higher node in $L_{x,s^1,x^2}$


\subsection{Optimization algorithm}

We apply an iterative optimization routine that starts on a low-resolution triangular prism grid $\T_0$ on which it solves for the discrete primal and dual variables,
resulting in discrete solutions $v^h_0\in S^{0,1}(\T_0)$ and $\phi^h_0\in S^{0,1}(\T_0)\times S^{0,1}(\T_0)\times S^{1}(\T_0)$.
According to some refinement criterion (to be discussed in \cref{sec:RefinementCriteria}) we then refine several elements of $\T_0$, resulting in a finer grid $\T_1$.
On this finer grid we again solve for the discrete primal and dual variables, resulting in $v^h_1,\phi^h_1$.
We then continue iteratively refining and solving on the grid, thereby producing a hierarchy $\T_0,\T_1,\ldots$ of grids with associated discrete solutions $v^h_k,\phi^h_k$, $k=1,2,\ldots$.

To solve the discrete saddle point problem on a given grid $\T_k$
we apply a standard primal-dual algorithm \cite{ChPo2011} in which we perform the projection onto the convex set $\hat\K^h(\T_k)$ via an iterative Dykstra routine \cite{BoDy1986}.
This projection is the computational bottleneck of the method (in terms of computation time as well as memory requirements),
and it is the main reason for using the tailored adaptive discretization introduced before.
In particular, note that the set of constraints in $\hat\K^h(\T)$ decomposes into subsets of constraints onto which the projection can be performed independently.
In detail, let $x^1,\ldots,x^p\in\R^2$ be the nodes of the two-dimensional simplex grid underlying the triangular prism grid and write
\begin{equation*}
(\Phi^1,\Phi^2)
=((\Phi^1,\Phi^2)_{x^1},\ldots,(\Phi^1,\Phi^2)_{x^p})
\end{equation*}
for $(\Phi^1,\Phi^2)_{x^i}=(\Phi^1_k,\Phi^2_k)_{k\in\hat L_{x,0,M}}$ and $\hat L_{x,s^1,s^2}=\{k\in\{1,\ldots,q''\}\,|\,N_k\in L_{x,s^1,s^1}\}$ the set of node indices belonging to $L_{x,s^1,s^2}$. Then
\begin{equation*}
\hat\K^h(\T)
=\left(\bigtimes_{i=1}^{p}\K_{x_i}\right)\times\K_s
\end{equation*}
for the convex sets
\begin{align*}
\K_x&=\left\{(\Phi^1_k,\Phi^2_k)_{k\in\hat L_{x,0,M}}\,\middle|\,\left|\textstyle\sum_{k\in\hat L_{x,s^1,s^2}} h(N_k)(\Phi^1_k,\Phi^2_k) \right| \leq \tau(|s^2-s^1|) \ \forall \ (x,s^1),(x,s^2)\in L_{x,0,M} \right\},\\
\K_s&=\{\Phi^s\in\R^{q'}\,|\,\Phi^s_k\geq0\,\forall k\}.
\end{align*}
so that one can project onto each $\K_{x^i}$ and $\K_s$ separately
(where the projection onto $\K_s$ is trivial and the projection onto each $\K_{x^i}$ is done via Dykstra's algorithm).
Note that this would change completely in the presence of $x$-hanging nodes.
Here, the set $x^1,\ldots,x^p$ of simplex grid nodes would also have to include the hanging nodes,
and as a consequence $\hat\K^h(\T)$ no longer decomposes into a Cartesion product of constraint sets $\K_{x^i}$
so that the projections can no longer be performed independently.

The overall procedure is presented in pseudocode in \cref{alg:FullFEAlgorithm}, using time steps $\tau,\sigma>0$ and an overrelaxation parameter $\theta$ from \cite{ChPo2011}
(throughout our numerical experiments we use $\theta=1$ as well as $\tau=\sigma=\frac1{L}$ for $L$ the Frobenius norm of the matrix $(M^1,M^2,M^s)$).


\begin{algorithm}
	\caption{Adaptive primal-dual algorithm for generalized branched transport problems}
	\label{alg:FullFEAlgorithm}
	\begin{algorithmic}
		\Function{OptimalTransportNetworkFE}{$u^{\text{start}}$,$\T_0$,$\tau$,$\sigma$,$\theta$,$\text{numRefinements}$} 
		\For{$i=0,\ldots,\text{numRefinements}$}
		\State assemble matrix $M=(M^1,M^2,M^s)$
		\If{$i=0$}
		\State $V^{0,0}=(1_{u^\text{start}}(N_1),\ldots,1_{u^\text{start}}(N_{q''}))$, $\Psi^{0,0}\equiv(\Phi^1,\Phi^2,\Phi^s)^{0,0}=0$
		\Else
		\State prolongate $(V^{i-1,\text{end}},\Psi^{i-1,\text{end}})$ on $\T_{i-1}$ to $(V^{i,0},\Psi^{i,0})$ on $\T_i$
		\EndIf
		\State $k\leftarrow0$
		\While{not converged}
		\State $\tilde\Psi^{i,k+1} = \tilde\Psi^{i,k}+\sigma M^\ast\bar{V}^{i,k}$
		\State compute the projection $\Psi^{i,k+1} = \pi_{\hat\K^h}(\tilde\Psi^{i,k+1})$ via
		\State \hspace*{15ex}separate projections onto sets $\K_{x^i},\K_s$ using Dykstra's algorithm
		\State $\tilde V^{i,k+1} = V^{i,k}-\tau M\Psi^{i,k+1})$
		\State compute the projection $V^{i,k+1} = \pi_{\C^h}(\tilde V^{i,k+1}$
		\State $\bar{V}^{i,k+1} = V^{i,k+1}+\theta(V^{i,k+1}-V^{i,k})$
		\State $k\leftarrow k+1$	
		\EndWhile 
		\If{$i < \text{numRefinements}$}
		\State refine grid $\T_i$ to $\T_{i+1}$
		\Else
		\State $V=V^{i,\text{end}}$, $(\Phi^1,\Phi^2,\Phi^s)=\Psi^{i,\text{end}}$
		\EndIf
		\EndFor
		\EndFunction
		\State \textbf{return} $V,\Phi^1,\Phi^2,\Phi^s$
	\end{algorithmic}
\end{algorithm}

\subsection{Refinement criteria} \label{sec:RefinementCriteria}

To decide which elements should be refined during the grid refinement in \cref{alg:FullFEAlgorithm} we use a combination (in our experiments, the maximum) 
of two heuristic criteria, which both seem to work reasonably well.
We define for each element $T\in\T_k$ a refinement indicator $\eta_T(v^h_k,\phi^h_k)$, depending on the solution $(v^h_k,\phi^h_k)$ of the discrete saddle point problem,
and we refine any element $T\in\T_k$ with
\begin{equation*}
\eta_T(v^h_k,\phi^h_k) \geq \lambda \max_{S\in\T_k} \eta_S(v^h_k,\phi^h_k)
\end{equation*}
for some fixed $\lambda\in(0,1)$.

The first choice of $\eta_T$ is based on the natural and intuitive idea to refine all those elements where the local gradient of the three-dimensional solution $v^h_k$ is high.
Indeed, $v^h_k$ approximates a continuous solution which we expect to be a characteristic function $1_u$
so that by finely resolving regions with high gradient $Dv^h_k$ we expect to better approximate $1_u$. Thus we define 
\begin{equation*}
\eta_T(v^h_k,\phi^h_k) = \frac{1}{\leb^3(T)}|Dv^h_k|(T').
\end{equation*}
Although this strategy is computationally cheap and easy to handle, gradient refinement only takes the current grid structure into account and neglects any information about the functional
(possibly leading to redundantly refined elements).

The second choice of $\eta_T$ is (an approximation of) the local primal-dual gap, that is, the contribution of each element to the global primal-dual gap
\begin{equation*}
\Delta(v^h_k,\phi^h_k)
=\G(v^h_k)-\D(\phi^h_k)
\geq0
\end{equation*}
associated with the strong duality from \cref{thm:strongDuality}.
Since $\Delta(v^h_k,\phi^h_k)=0$ implies that $(v^h_k,\phi^h_k)$ are the global solution of the saddle point problem,
it is natural to refine the grid in those regions where the largest contribution to the duality gap occurs.
This contribution can be calculated as follows,
\begin{align*}
\Delta(v^h_k,\phi^h_k)
&=\sup_{\phi\in\hat\K}\int_{\overline\Omega\times[0,M)}\phi\cdot\d Dv^h_k-\min_{v\in\C}\int_{\overline\Omega\times[0,M)}\phi^h_k\cdot\d Dv\\
&=\sup_{\phi\in\hat\K}\int_{\overline\Omega\times[0,M)}(\phi-\phi^h_k)\cdot Dv^h_k\,\d x\,\d s-\min_{v\in\C}\int_{\overline\Omega\times[0,M)}\phi^h_k\cdot\d D(v-v^h_k)\\
&=\sup_{\phi\in\hat\K}\int_{\overline\Omega\times[0,M)}(\phi^x-(\phi^h_k)^x)\cdot D_xv^h_k\,\d x\,\d s+\max_{v\in\C}\int_{\overline\Omega\times[0,M)}(v-v^h_k)\dive\phi^h_k\,\d x\,\d s.
\end{align*}
While the maxizing $v^{\mathrm{opt}}$ can readily be calculated as $v^{\mathrm{opt}}(x,s)=\max\{0,\mathrm{sign}(\dive\phi^h_k)\}$,
the supremum has no analytical expression and needs to be evaluated numerically.
We approximate it by refining $\T_k$ uniformly to some grid $\tilde\T_k$ and then calculating
\begin{equation*}
\phi^{\mathrm{opt}}=\argmax_{\phi\in\hat\K^h(\tilde\T_k)}\int_{\overline\Omega\times[0,M)}\phi^x\cdot D_xv_k^h\,\d x\,\d s
\in S^{0,1}(\T_k)\times S^{0,1}(\T_k)\times\{0\}.
\end{equation*}
Note that this latter maximization can be independently performed for the function values at nodes with different $x$-coordinates and thus is very fast.
We then set the refinement indicator as
\begin{equation*}
\eta_T(v^h_k,\phi^h_k)
=\int_T((\phi^{\mathrm{opt}})^x-(\phi^h_k)^x)\cdot D_xv^h_k\,\d x\,\d s+\int_T(v^{\mathrm{opt}}-v^h_k)\dive\phi^h_k\,\d x\,\d s.
\end{equation*}
Since $\phi^{\mathrm{opt}}$ is only an approximation of the true minimizer, $\sum_{T\in\T_k}\eta_T(v^h_k,\phi^h_k)\leq\Delta(v^h_k,\phi^h_k)$ is an approximation of the duality gap from below.
Note that the summand $\int_T(v^{\mathrm{opt}}-v^h_k)\dive\phi^h_k\,\d x\,\d s$ is nonnegative,
while $\int_T((\phi^{\mathrm{opt}})^x-(\phi^h_k)^x)\cdot D_xv^h_k\,\d x\,\d s$ in principle may have either sign.
However, at least we have $\sum_{x^i\in\pi_x(T)}\int_T((\phi^{\mathrm{opt}})^x-(\phi^h_k)^x)\cdot D_xv^h_k\,\d x\,\d s\geq0$ for all simplex grid nodes $x^i$
(where $\pi_x:\R^3\to\R^2$ shall be the projection onto the first two coordinates)
so that $\eta_T(v^h_k,\phi^h_k)$ may well serve as a local refinement indicator.

\subsection{Results}

We implemented the algorithm described above in C++, where the grid and corresponding finite element classes are based on the QuocMesh library \cite{Quocmesh}. For our experiments we pick the branched transport and urban planning transportation costs $\tau$ from \cref{ex:BTandUP}. 

To begin with, we test the reliability of the method by comparing its results with the true solution in a simple symmetric setting
in which the optimal transport network can actually be calculated by hand.
This setting has four evenly spaced point sources of equal mass at the top side of the rectangular domain $\Omega=[0,1]^2$
and four evenly spaced point sinks of same mass exactly opposite.
Due to the high symmetry there are only a handful of possible graph topologies whose vertex positions can explicitly be optimized.
For both branched transport and urban planning we test a range of parameters in order to explore multiple different topologies.
\Cref{fig:EnergyPlotBTAdaptive,fig:EnergyPlotUPAdaptive} show that in each case the algorithm converged to the correct solution
except for one parameter setting close to a bifurcation point where the optimal network topology changes.
In that setting our algorithm returned a convex combination of functions $1_u$ corresponding to two different topologies,
which numerically both seem to be of sufficiently equal optimality so that the algorithm converges to their convex combination (compare \cref{thm:propertiesConvex}\eqref{item:nonbinary}).
This is in fact a slight improvement over the result in \cite{BrRoWi2018}, where we performed exactly the same experiment,
only using a standard finite difference discretization at much lower resolution.
For that discretization and resolution the algorithm actually converged to the wrong topology,
which was better aligned with the grid and therefore advantageous at the given resolution.
With our new discretization we achieve a higher resolution, enabling the algorithm to move away from that erroneous topology.
It seems that more grid refinement would be necessary to recover the true solution, however,
to make the results for all parameters comparable we chose the same number of refinements throughout \cref{fig:EnergyPlotBTAdaptive,fig:EnergyPlotUPAdaptive}.

Note that the reliability of the algorithm is not obvious a priori since an adaptive refinement may in principle lead to discretization artefacts,
giving preference to material fluxes through highly resolved areas over fluxes through coarsly discretized areas, in which the discretization error produces artificial additional costs.

\begin{figure}
    \setlength\unitlength\textwidth
    \begin{picture}(1,.7)
	\put(0,0){\includegraphics[width=1\unitlength]{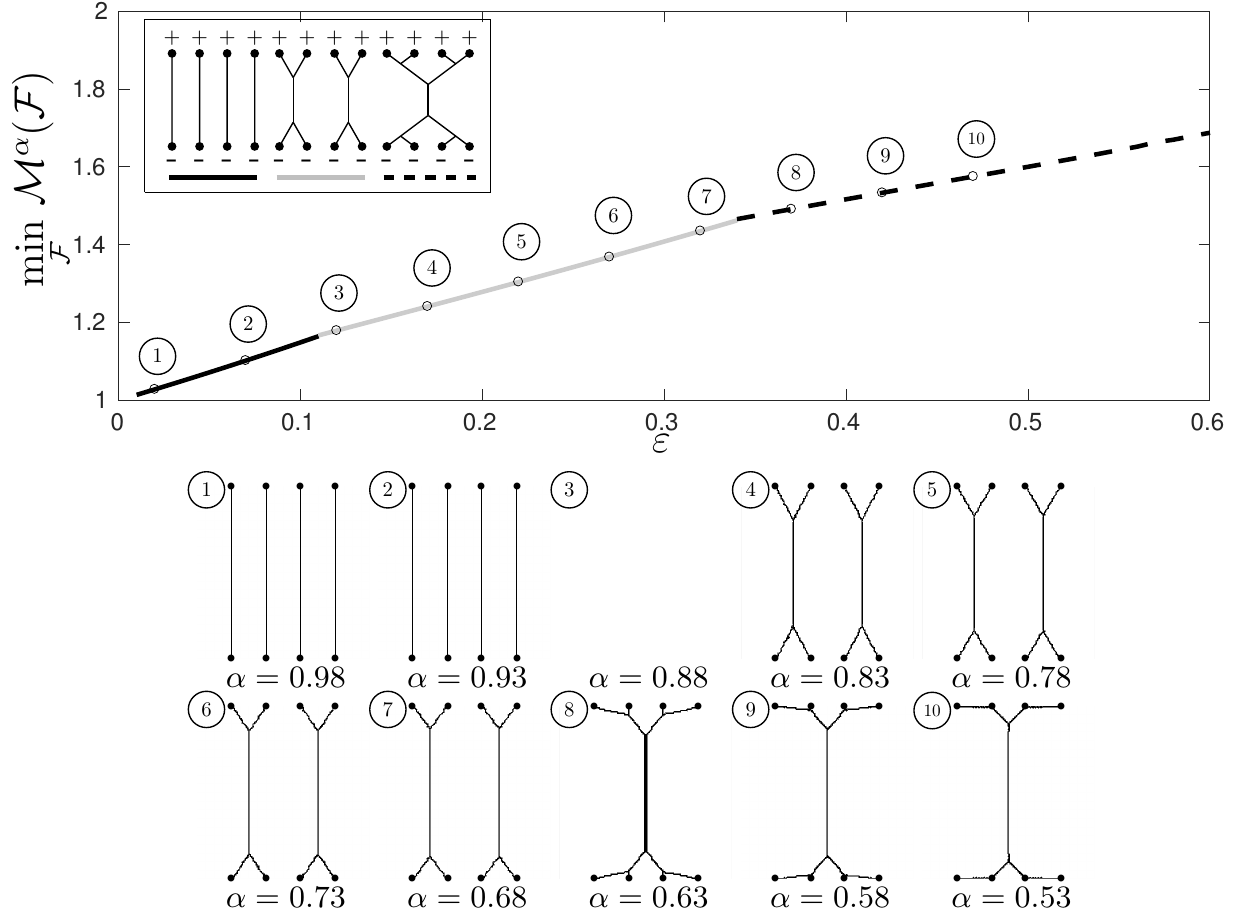}}
	\put(.51,.35){\color{white}\rule{3ex}{3.65ex}}
	\put(0,.5){\color{white}\rule{5.7ex}{18ex}}
	\put(.02,.52){\rotatebox{90}{$\displaystyle\min_{\flux\in\A_\flux}\E(\flux)$}}
	\put(.5,.37){$1-\alpha$}
	\end{picture}
	\caption{Parameter study for branched transport with transportation cost $\tau^{\mathrm{bt}}(m)=m^\alpha$. Top: Plot of the manually and the numerically computed minimal energy for different values of $1-\alpha$. The line type indicates the optimal network topology.	Bottom: Numerically computed optimal transport networks for evenly spaced values of $\alpha$ in the same range (if the numerical solution is $v$, we show the support of the gradient of its projection onto the $x$-plane). The numerically obtained network topologies match the predicted ones except for example \textcircled{\small 3}, where the three-dimensional solution is not binary, but a convex combination of the binary solutions to two different topologies\notinclude{ consisting of four straight lines and three trees (two straight lines left and right and a single tree connecting the upper and lower middle points)}.}
	\label{fig:EnergyPlotBTAdaptive}
\end{figure}
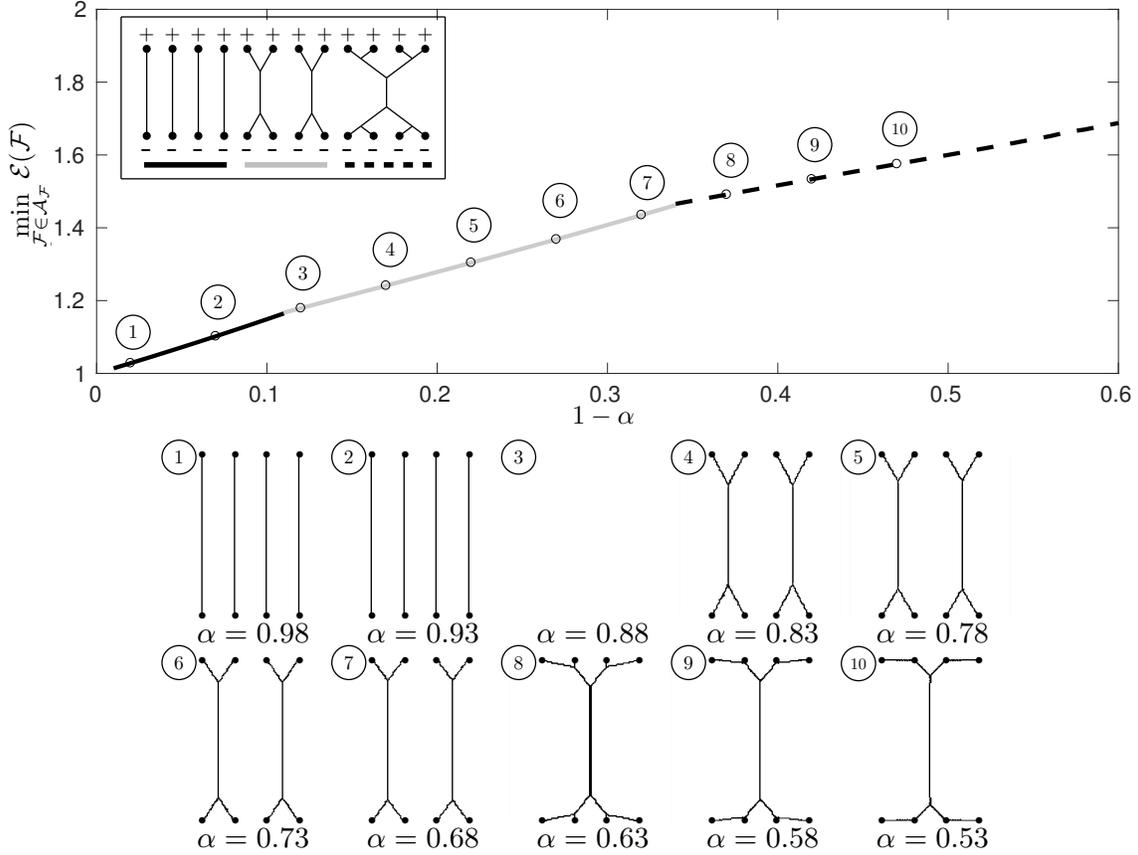

\begin{figure}
    \setlength\unitlength\textwidth
    \begin{picture}(1,.7)
	\put(0,0){\includegraphics[width=1\unitlength]{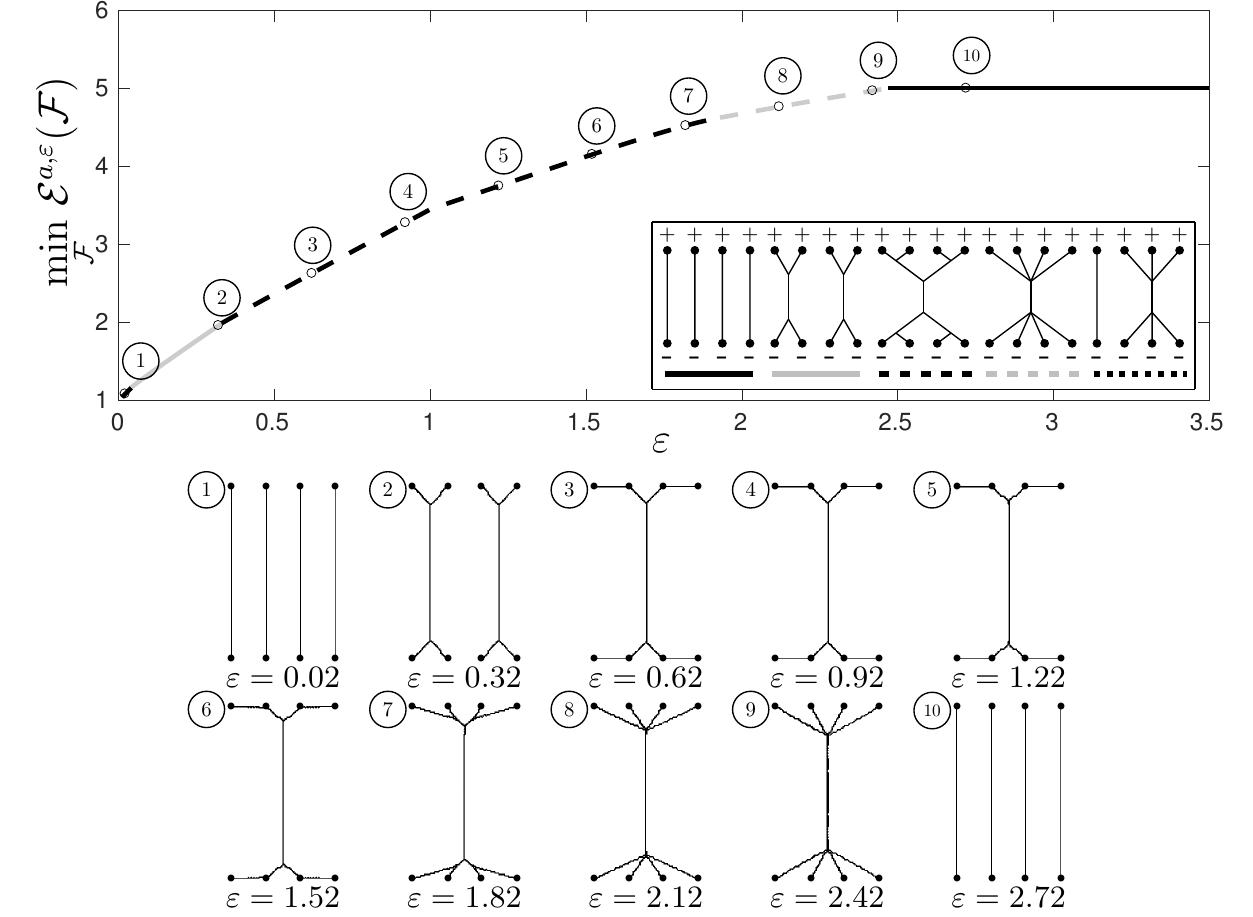}}
	\put(.51,.35){\color{white}\rule{3ex}{3.65ex}}
	\put(0,.5){\color{white}\rule{8ex}{18ex}}
	\put(.03,.52){\rotatebox{90}{$\displaystyle\min_{\flux\in\A_\flux}\E(\flux)$}}
	\put(.53,.37){$b$}
	\multiput(.18,.0060)(0,.175){2}{\multiput(0,0)(.145,0){5}{\color{white}\rule{1.5ex}{1.5ex}\color{black}\hspace{-1.5ex}\raisebox{.2ex}{$b$}}}
	\end{picture}
	\caption{Parameter study for urban planning with transportation cost $\tau^{\mathrm{up}}(m)=\min\{am,m+b\}$ for $a=5$ and varying $b$. Illustration as in \cref{fig:EnergyPlotBTAdaptive}.}
	\label{fig:EnergyPlotUPAdaptive}
\end{figure}

At this point we would also like to mention that in \cite{BrRoWi2018} we obtained one simulation result for urban planning
with the same four sources and sinks as in \cref{fig:EnergyPlotUPAdaptive} (but different parameter values)
which was not binary and which we assumed to be a manifestation of the convex relaxation being not tight.
However, it turns out that the result was again just a convex combination of two global minimizers,
namely the right-most topology in \cref{fig:EnergyPlotUPAdaptive} and its mirror image
(which just happen to be never optimal for the parameters in \cref{fig:EnergyPlotUPAdaptive}).

Next we repeat the other numerical simulations from \cite{BrRoWi2018} which require transport networks of much more complex branching structure
and which due to a lack of resolution could hardly be resolved in \cite{BrRoWi2018}
(in fact, the smallest obtained network branches were at the order of the discretization width, and all network branches were visibly distorted by the pixel grid).
\Crefrange{fig:Results16To16FE}{fig:Results32CircularFE} show simulation results for these configurations with much more satisfying accuracy
at which all branches are clearly resolved.

\begin{figure}
    \setlength\unitlength\textwidth
    \begin{picture}(1.05,.48)
	\put(0,0){\includegraphics[width=1.05\unitlength,trim=0 20 0 20,clip]{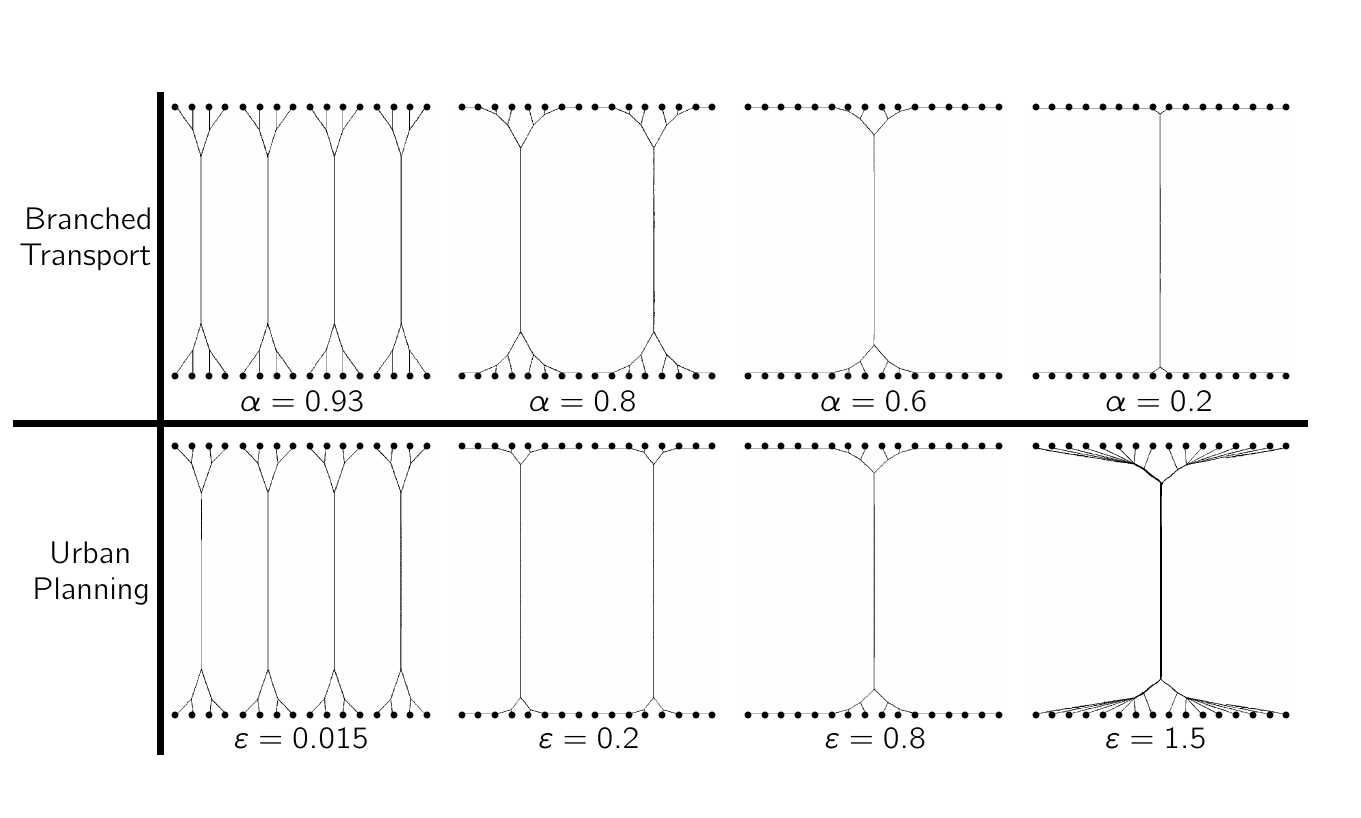}}
	\multiput(.18,-.003)(.236,0){2}{\color{white}\rule{1.5ex}{1.5ex}\color{black}\hspace{-1.5ex}\raisebox{.2ex}{$b$}}
	\multiput(.632,-.003)(.217,0){2}{\color{white}\rule{1.5ex}{1.5ex}\color{black}\hspace{-1.5ex}\raisebox{.2ex}{$b$}}
	\end{picture}
	\caption{Numerical optimization results for transport from 16 almost evenly spaced point sources to 16 point sinks of the same mass ($a=5$ for the urban planning results).}
	\label{fig:Results16To16FE}
\end{figure}

\begin{figure}
    \setlength\unitlength\textwidth
    \begin{picture}(1.05,1.05)
	\put(0,0){\includegraphics[width=1.05\textwidth,trim=0 11 0 20,clip]{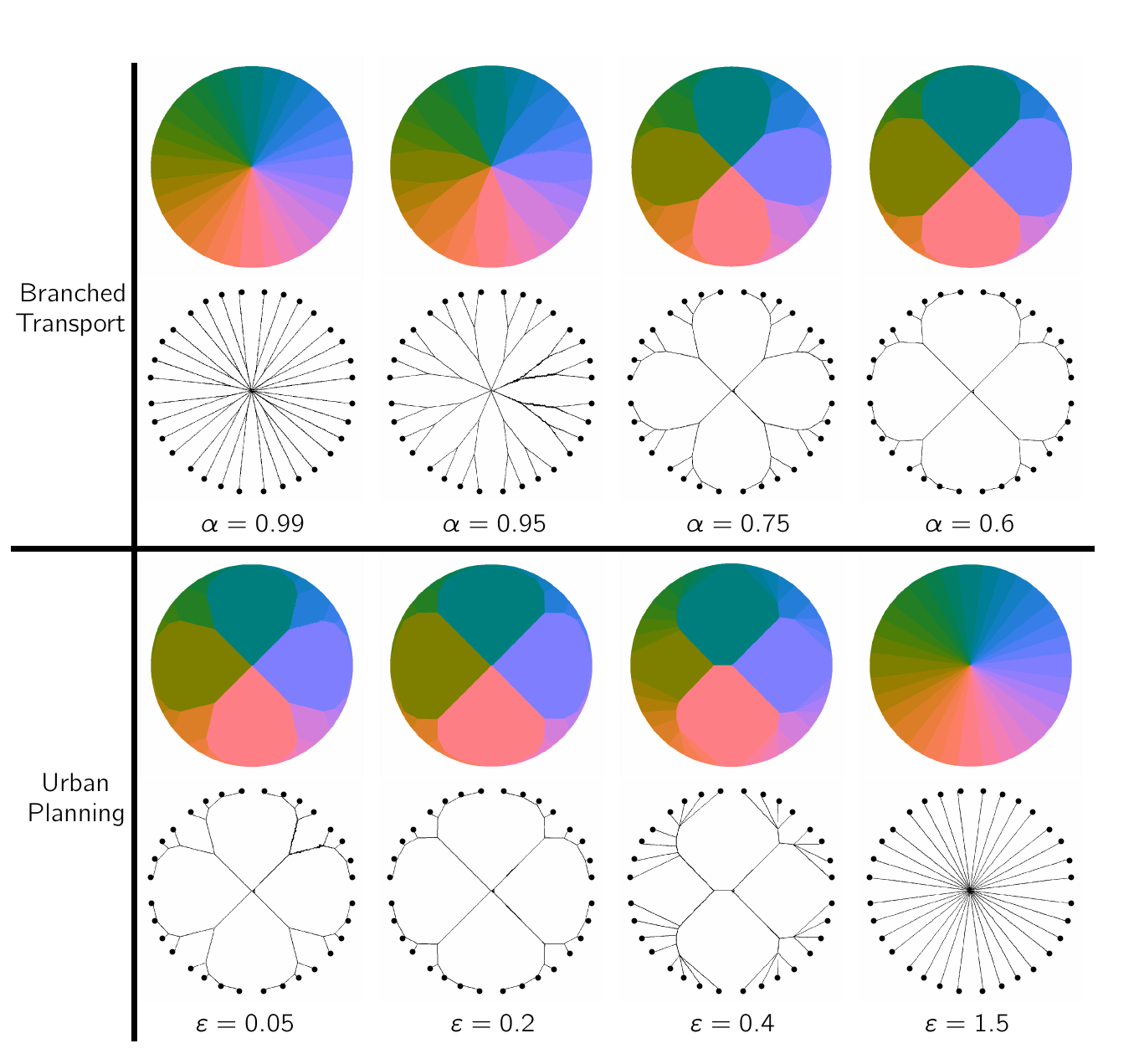}}
	\multiput(.18,-.003)(.236,0){2}{\color{white}\rule{1.5ex}{1.5ex}\color{black}\hspace{-1.5ex}\raisebox{.2ex}{$b$}}
	\multiput(.632,-.003)(.217,0){2}{\color{white}\rule{1.5ex}{1.5ex}\color{black}\hspace{-1.5ex}\raisebox{.2ex}{$b$}}
	\end{picture}
	\caption{Numerical optimization results for transport from a central source point to 32 almost evenly spaced point sinks of equal mass on a concentric circle ($a=5$ for the urban planning results). 
	Using a periodic colour-coding, we show the images $u$, whose lifting $1_u$ is the numerical solution, as well as the support of their gradient underneath, which represents the transport network.}
	\label{fig:Results32CircularFE}
\end{figure}

In these rather symmetric example settings we slightly broke the symmetry by perturbing the even spacing of sources and sinks,
since otherwise there would be multiple global optimal transport networks, a convex combination of which would be returned by our algorithm.
To be able to have a source point within the domain $\Omega$ in \cref{fig:Results32CircularFE} (recall that $\mu_+,\mu_-$ should lie on $\partial\Omega$) we employ the following trick:
we connect the centre source with the boundary $\partial\Omega$ by a (straight) line across which we enforce the variables $v$ and $\phi$ to be discontinuous with
\begin{equation*}
v^-(x,s)=v^+(x,s+M),\quad
\phi^-(x,s)=\phi^+(x,s+M)
\end{equation*}
for $M=\|\mu_+\|_\M=\|\mu_-\|_\M$.
Essentially this means that we take the range of the two-dimensional images $u$ (corresponding to the mass fluxes) to be an infinite covering of $[0,M)$ with fibres $r+M\mathbb Z$.

We finally discuss the gain in computational efficiency by the new adaptive discretization.
We already saw before that the adaptive discretization allows to produce a quality of the transport networks that goes far beyond a standard discretization.
At the same time, the computational cost decreases.
\Cref{fig:adaptiveGrid} illustrates, for a simple example that can readily be visualized, the reason for the enhanced efficiency, the underlying adaptive grid refinement near the network branches.
\Cref{tab:speedup,fig:speedup} quantify the speedup of going from a standard uniform discretization to the adaptive one
(for the same configuration as in \cref{fig:EnergyPlotBTAdaptive} with $\alpha=0.5$), which quickly reaches orders of magnitude.

\begin{figure}
	\setlength{\unitlength}{\linewidth}
	\begin{picture}(1.2,.35)
	\put(0,0){\includegraphics[width=0.3\textwidth]{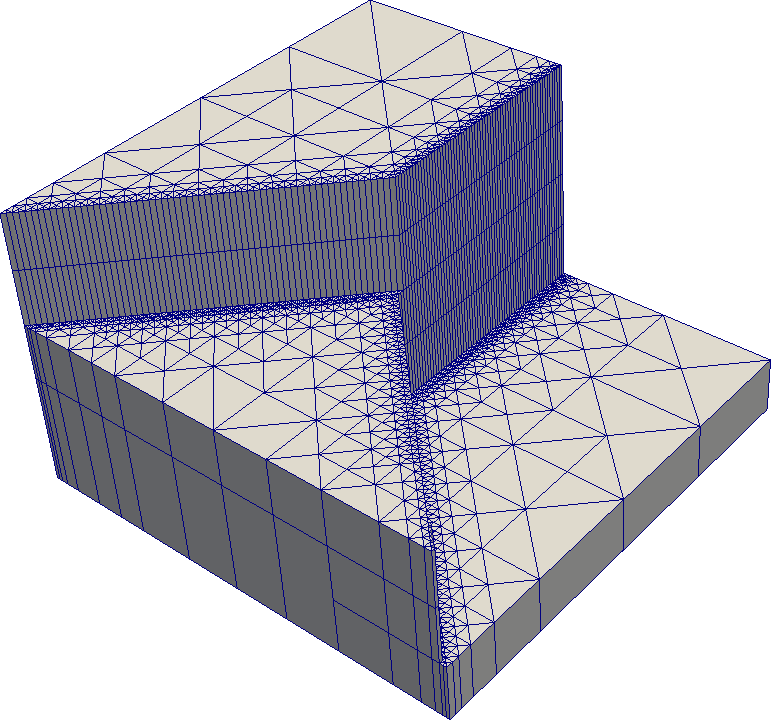}}
	\put(0.37,0.02){\includegraphics[width=0.25\textwidth]{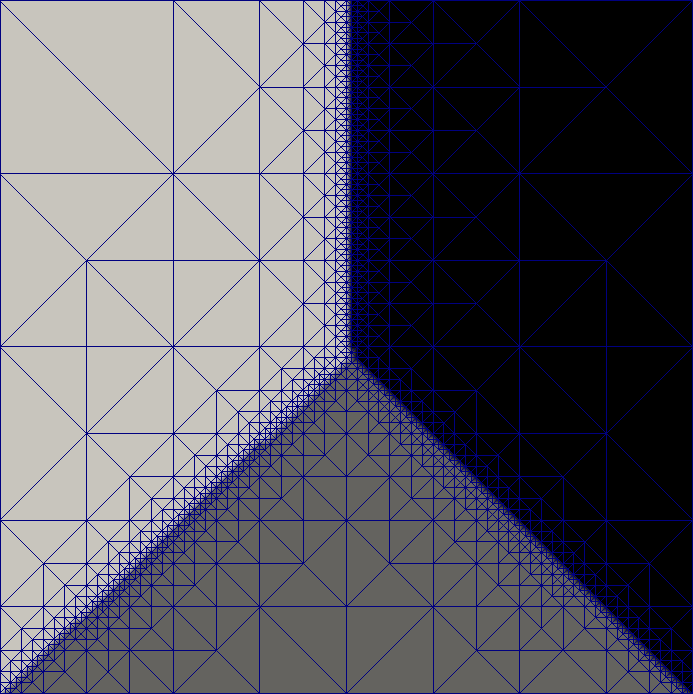}}
	\put(0.7,0.03){\includegraphics[width=0.22\textwidth]{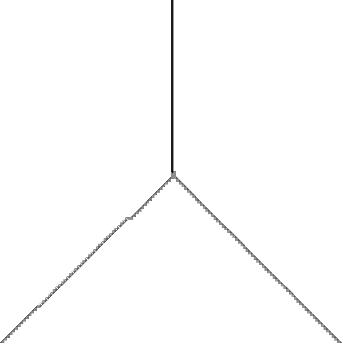}}
	\put(0.812,0.246){\circle*{0.012}}
	\put(0.703,0.03){\circle*{0.012}}
	\put(0.92,0.03){\circle*{0.012}}
	\put(0.802,0.26){$+$}
	\put(0.69,0.005){$-$}
	\put(0.91,0.005){$-$}
	\end{picture}
\caption{Optimal network for branched transport from one point mass at the top of $\Omega=[0,1]^2$ to two equal sinks at the bottom corners of $\Omega$.
Left: Profile of the three-dimensional discrete solution $v^h$ (the displayed surface shows the $\frac12$-level set of $v$ with finite element boundaries indicated in blue).
Middle: Two-dimensional image obtained by projecting $v^h$ onto the $x$-plane (the element boundaries of the underlying two-dimensional simplex grid are shown in blue).
Right: Optimal network structure, given by the support of the image gradient.}
\label{fig:adaptiveGrid}
\end{figure}

\begin{table}
	\centering
	\scalebox{0.76}{
		\begin{tabular}{|c|cccc|cccccc|}
			\hline 
			& \multicolumn{4}{|c|}{Uniform} & \multicolumn{6}{|c|}{Adaptive} \\
			\hline
			x/s & numEls & numDofs & time & pd gap & numEls & numDofs & \%Els & \%Dofs & time & pd gap \\
			\hline 
			4/2 & 2048 & 1445 & 14 sec. & 0.0069 & 2048 & 1445 & 100 & 100 & 14 sec. & 0.0069 \\
			5/3 & 16384 & 9801 & 96 sec. & 0.0192 & 7111 & 4576 & 43.4 & 46.7 & 44 sec. & 0.0101 \\
			6/4 & 131072 & 71825 & 855 sec. & 0.0165 & 30961 & 18800 & 23.6 & 26.2 & 184 sec. & 0.0431 \\
			7/5 & 1048576 & 549153 & 20014 sec. & 0.0013 & 91391 & 53596 & 8.7 & 9.8 & 632 sec. & 0.0027 \\
			8/6 & 8388608 & 4293185 & 224221 sec. & 0.0047 & 146825 & 84749 & 1.7 & 2.0 & 1405 sec. & 0.0019 \\
			9/7 & - & - & - & - & 295227 & 167030 & 0.4 & 0.5 & 3438 sec. & 0.0008 \\
			10/8 & - & - & - & - & 667289 & 370570 & 0.1 & 0.1 & 9767 sec. & 0.0003 \\
			\hline 
		\end{tabular}
	}
\caption{Comparison between branched transport network simulations on a uniform and an adaptive grid.
The first column refers to the $x$- and $s$-level of the uniform grid and the highest local $x$- and $s$-level of the adaptive grid
(the $x$- and $s$-level of an element is the number of $x$- and $s$-bisections necessary to obtain the element, starting from an element of the same size as the computational domain).
The table shows the number of elements, of degrees of freedom in the variable $v^h$, the runtime, and the calculated primal-dual gap at the end.
For the adaptive simulation, the relative numer of elements and degrees of freedom compared to the uniform simulation is also shown as a percentage.
All adaptive simulations start at a uniform grid of $x$-level $4$ and $s$-level $2$.
The experiments on a uniform grid of the highest levels are omitted due to their infeasible runtime and memory consumption.}
\label{tab:speedup}
\end{table}

\begin{figure}
	\begin{minipage}[b]{0.45\linewidth}
		\centering
		\includegraphics[width=1.1\textwidth]{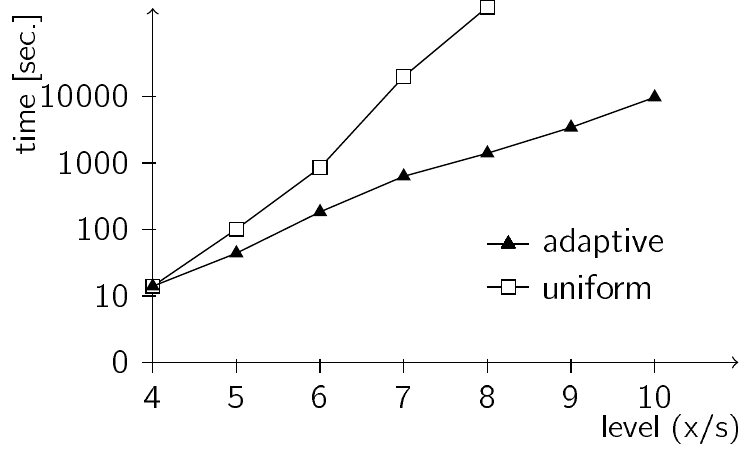}
	\end{minipage} 
	\hspace{1cm}
	\begin{minipage}[b]{0.45\linewidth}
		\centering
		\includegraphics[width=1\textwidth]{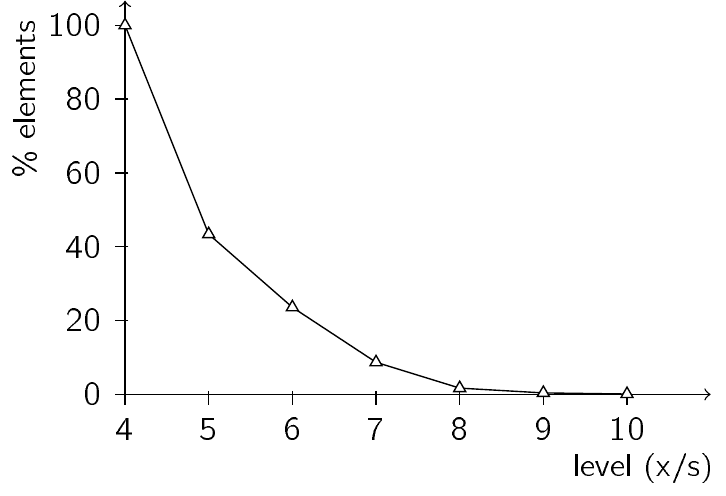}
	\end{minipage} 
\caption{Runtime and relative number of elements in a simulation on an adaptive versus a uniform grid from \cref{tab:speedup}.}
\label{fig:speedup}
\end{figure}


\section{Discussion} 

We shed more light on the relation between two-dimensional generalized branched transport and corresponding convex optimization problems obtained via functional lifting.
In particular, it is now clear that those problems are indeed equivalent up to a relaxation step whose tightness is expected, but not known.
With a tailored adaptive finite element discretization, this relation could now be leveraged to solve two-dimensional generalized branched transport problems.

A seeming disadvantage of the functional lifting approach lies in the fact that the given material source and sink $\mu_+,\mu_-$ need to be supported on the computational domain boundary.
This deficiency can be overcome by a trick similar to the one of \cref{fig:Results32CircularFE}, introduced in \cite{BoOrOu2016}.
To this end one fixes an initial backward mass flux $\flux_-$ from $\mu_-$ to $\mu_+$.
Taking now any mass flux $\flux$ from $\mu_+$ to $\mu_-$, the joint flux $\flux+\flux_-$ has zero divergence and can thus be translated into the gradient of an image.
During the image optimization or the corresponding lifted convex optimization one just has to ensure by constraints that the backward mass flux stays fixed and is not changed
(and also one has to adapt the cost functional so as to neglect the cost of $\flux_-$ and to prevent artificial cost savings that may come about by aggregating part of $\flux$ with $\flux_-$).

%

A true disadvantage, though, of the approach is that it is inherently limited to two space dimensions.
Indeed, it exploits that in two space dimensions the one-dimensional network structures also have codimension $1$ and thus can be interpreted as image gradients.
However, the two-dimensional case is of importance in various settings such as logistic problems, public transport networks, river networks or leaf venation, to name but a few examples.

Compared to graph-based methods, computation times of our approach are of course much longer, however, our approach is guaranteed to yield a global minimizer.
Nevertheless, heuristic topology optimization procedures on graphs seem to result in networks of almost the same quality.
It is conceivable that a combination of both approaches may increase efficiency while maintaining the guarantee of a global minimum.


\section{Acknowledgement}
The work was supported by the Alfried Krupp Prize for Young University Teachers awarded by the Alfried Krupp von Bohlen und Halbach-Stiftung as well as by the Deutsche Forschungsgemeinschaft (DFG, German Research Foundation) under Germany's Excellence Strategy through the Cluster of Excellence ``Mathematics M\"unster: Dynamics -- Geometry -- Structure'' (EXC 2044 - 390685587) at the University of M\"unster and through the DFG-grant WI 4654/1-1 within the Priority Program 1962.

\bibliographystyle{plain}
\bibliography{AdaptiveBranchedTransport}

\end{document}